\documentclass[12pt,reqno]{amsart}

\author{Richard Gottesman}
\address{Department of Mathematics and Statistics\\ Queen's University\\ Kingston, ON K7L3N6 }
\email{richard.b.gottesman@gmail.com}

\title{The arithmetic of vector-valued modular forms on $\Gamma_{0}(2)$}
\usepackage{amsmath,amssymb,amsthm,url,pdfsync}
\usepackage{geometry}
\usepackage{thmtools}
\usepackage{thm-restate}
\usepackage{hyperref}
\usepackage[capitalise]{cleveref}
\newcommand{\nin}{\noindent}
\newcommand{\Z}{\textbf{Z}}
\newcommand{\Q}{\textbf{Q}}
\newcommand{\C}{\textbf{C}}
\newcommand{\N}{\textbf{N}}
\newcommand{\Ind}{\textrm{Ind}}
\allowdisplaybreaks

\declaretheorem[name=Theorem, numberwithin=section]{thm}

\newtheorem{cor}[thm]{Corollary}
\newtheorem{assumption}[thm]{Assumption}
\newtheorem{lemma}[thm]{Lemma}
\newtheorem{proposition}[thm]{Proposition}

\newtheorem{definition}[thm]{Definition}

\newtheorem{remark}[thm]{Remark}

\newcommand{\Poincare}{Poincar\'e}

\begin{document}
\renewcommand{\labelenumi}{(\roman{enumi})}
\def\d{\mathrm{d}}
\def\e{e}

\newcommand{\leg}[2]{\genfrac{(}{)}{}{}{#1}{#2}}
\makeatletter
\def\moverlay{\mathpalette\mov@rlay}
\def\mov@rlay#1#2{\leavevmode\vtop{%
   \baselineskip\z@skip \lineskiplimit-\maxdimen
   \ialign{\hfil$\m@th#1##$\hfil\cr#2\crcr}}}
\newcommand{\charfusion}[3][\mathord]{
    #1{\ifx#1\mathop\vphantom{#2}\fi
        \mathpalette\mov@rlay{#2\cr#3}
      }
    \ifx#1\mathop\expandafter\displaylimits\fi}
\makeatother

\newcommand{\cupdot}{\charfusion[\mathbin]{\cup}{\cdot}}
\newcommand{\bigcupdot}{\charfusion[\mathop]{\bigcup}{\cdot}}

\begin{abstract}
\nin Let $\rho$ denote an irreducible two-dimensional representation of $\Gamma_{0}(2).$ The collection of vector-valued modular forms for $\rho$, which we denote by $M(\rho)$, form a graded and free module of rank two over the ring of modular forms on $\Gamma_{0}(2)$, which we denote by $M(\Gamma_{0}(2)).$ For a certain class of $\rho$, we prove that if $Z$ is any vector-valued modular form for $\rho$ whose component functions have algebraic Fourier coefficients then the sequence of the denominators of the Fourier coefficients of both component functions of $Z$ is unbounded. Our methods involve computing an explicit basis for $M(\rho)$ as a $M(\Gamma_{0}(2))$-module. We give formulas for the component functions of a minimal weight vector-valued form for $\rho$ in terms of the Gaussian hypergeometric series $_{2}F_{1}$, a Hauptmodul of $\Gamma_{0}(2)$, and the Dedekind $\eta$-function.
 \end{abstract}
\maketitle 
\section{Introduction} 
\nin The arithmetic of the Fourier coefficients of vector-valued modular forms for a representation of the modular group $\Gamma := \textrm{SL}_{2}(\Z)$ have been intensively studied by Cameron Franc, Chris Marks, and Geoff Mason. One of the motivations for their work is the unbounded denominator conjecture of Atkin and  Swinnerton-Dyer \cite{atkin}. Atkin and Swinnerton-Dyer gave examples of modular forms on noncongruence subgroups whose Fourier coefficients have unbounded denominators. It is an open problem to show that modular forms on noncongruence subgroups have unbounded denominators. Progress on this problem has been made by Kurth and Long \cite{long2008modular}, \cite{kurth2008modular}, and by Li and  Long \cite{li2012fourier}. 
In \cite{mason2012}, Mason considered the unbounded denominator problem for vector-valued modular forms. 
Mason \cite{mason2012} proved that for all but a finite number of two-dimensional irreducible representations $\rho$ of $\Gamma$, every vector-valued modular form for $\rho$ whose Fourier coefficients are algebraic numbers has the property that the denominators of the Fourier coefficients of each of its component functions is unbounded. Marks \cite{chrismarks} has proven the analogous result for all but a finite number of three-dimensional representations $\rho$ of $\Gamma.$ Franc and Mason \cite{franc2013fourier} solved a modular linear differential equation to prove the same result for all two-dimensional representations $\rho$ such that $\textrm{ker } \rho$ is noncongruence. The main theorem in this paper is the proof of such an unbounded denominator result for a certain class of representations of $\Gamma_{0}(2).$ \\

\nin In the course of proving this result, we solve the general monic modular linear differential equation of order two on $\Gamma_{0}(2).$ Modular linear differential equations have been studied by Kaneko \cite{kanekomodular}, Kaneko and Koike \cite{kanekokoike}, Kaneko, Nagatomo, and Sakai \cite{kaneko2017}, Sebbar and Sebbar \cite{sebbar}, and Kaneko and Zagier \cite{kanekozagier}. Our analysis of the Fourier coefficients of a basis of solutions to the relevant modular linear differential equation boils down to understanding the number-theoretic properties of the coefficients of the hypergeometric series $_{2}F_{1}$ at arguments in a quadratic field.  Hong and Wang \cite{hong} adopt a $p$-adic perspective to study the arithmetic of the coefficients of the hypergeometric series $_{2}F_{1}$ at arguments in a quadratic field. The arithmetic of the coefficients of $_{2}F_{1}$ with rational parameters are examined in the works of Dwork \cite{dwork2}, \cite{dwork1}, Christol \cite{christol}, Franc, Gannon, and Mason \cite{franc2018unbounded}, and Franc, Gill, Goertzen, Pas, and Tu \cite{francdensities}. \\

\nin Let $\mathfrak{H}$ denote the complex upper-half plane. For each $k \in \Z$, we shall define an action of $\Gamma$ on holomorphic functions of $\mathfrak{H}.$ Let $t$ denote a positive integer and let  $F: \mathfrak{H} \rightarrow \C^{t}$ denote a holomorphic function. If $\gamma =  \left[ {\begin{array}{cc}
   a & b \\
   c & d \\
\end{array} } \right] \in \Gamma $ then we define 
\begin{equation} 
\label{eqn:slash}
 F|_{k}  \gamma (\tau) :=  (c \tau + d)^{-k} F\left(\frac{a \tau + b}{c \tau + d}\right). 
\end{equation}
\nin Let $H$ denote a finite index subgroup of $\Gamma$ and let $\rho$ denote a finite-dimensional complex representation of $H.$
\begin{definition} A \textbf{vector-valued modular form }$F$ of weight $k$ with respect to $\rho$ is a holomorphic function $F: \mathfrak{H} \rightarrow \C^{\textrm{dim } \rho}$ which is also holomorphic at all of the cusps of \newline 
$H \backslash (\mathfrak{H} \bigcup \mathbb{P}^{1}(\Q))$ and such that for all  $\gamma \in H$, 
\begin{equation} \label{eqn: transformation} F|_{k} \gamma =  \rho(\gamma)F.\; \; \; \;  \end{equation}  \end{definition}

\nin The statement that $F$ is holomorphic at all of the cusps means that for each $\gamma \in \Gamma$, 
$F|_{k} \gamma$ has a holomorphic $q$-expansion. The notion of a holomorphic $q$-expansion for a vector-valued modular form is more intricate than in the scalar-valued case. 
A detailed description is given in Section \ref{derivative}.  \\

 \nin  We denote the collection of all weight $k$ vector-valued modular forms with respect to $\rho$ by $M_{k}(\rho).$  For each $k \in \Z$, $M_{k}(\rho)$ is a finite-dimensional $\C$-vector space. We let $M(\rho) := \bigoplus_{k \in \Z} M_{k}(\rho).$ Let $M_{t}(H)$ denote the collection of all holomorphic weight $t$ modular forms on $H$ and let $M(H) := \bigoplus_{t \in \Z} M_{t}(H).$ If $m \in M_{t}(H)$ and if $F \in M_{k}(\rho)$ then $mF \in M_{k +t}(\rho).$ In this way, $M(\rho)$ has the structure of a $\Z$-graded $M(H)$-module. The structure of $M(\rho)$ as a $M(H)$-module is of fundamental importance in this paper. If $H = \Gamma$ then the $M(H)$-module structure of $M(\rho)$ is completely understood:

\begin{thm} \label{MarksMason} Let $\rho$ denote a representation of $\Gamma$. Then $M(\rho)$ is a free 
$M(\Gamma)$-module whose rank equals the dimension of $\rho$. 
\end{thm} 

\nin Theorem \ref{MarksMason} was proven by Chris Marks and Geoff Mason using vector-valued \Poincare \; series \cite{marksmason}, by Terry Gannon using a Riemann-Hilbert perspective \cite{gannon2014theory}, and by Luca Candelori and Cameron Franc using an algebro-geometric approach \cite{francfreemodule}.  In unpublished work, Mason has shown that $M(\rho)$ need not be free as a $M(H)$-module for subgroups $H$ 
of finite index in $\Gamma$ -- see \cite{structure} for a general discussion of questions of this nature, and an example that illustrates that $M(\rho)$ need not even be projective over $M(H).$ Nevertheless, $M(\rho)$ is a free $M(H)$-module if $H = \Gamma_{0}(2).$ This result follows immediately from the fact that $M(\Gamma_{0}(2))$ is generated as a $\C$-algebra by two modular forms which are algebraically independent and the following theorem: 

\begin{thm} \label{free} Let $H$ denote a finite index subgroup of $\Gamma$ and let $\rho$ denote a representation of $H$. Suppose that there exist modular forms $X$ and $Y$ in $M(H)$ which are algebraically independent such that $M(H) = \C[X,Y]$. Then $M(\rho)$ is a free $M(H)$-module whose rank equals the dimension of $\rho$. Moreover, there exists a $M(H)$-basis for $M(\rho)$ which consists of elements of $M(\rho)$ which are vector-valued modular forms for $\rho.$
\end{thm} 

\nin The author uses commutative algebra to give a proof of Theorem \ref{free} in \cite{CohenM}. We now explain how this theorem also follows from the work of  Candelori and Franc \cite{structure}. In \cite{structure}, Candelori and Franc define geometrically weighted vector-valued modular forms. They prove that if $H$ is a Fuchsian group of genus zero with at most two elliptic points then the collection of geometrically weighted vector-valued modular forms for $\rho$ is a free module over the ring of geometrically weighted modular forms for $H.$ If $H$ satisfies the hypothesis of Theorem \ref{free} then the collection of geometrically weighted vector-valued modular forms is equal to $M(\rho)$ and the collection of geometrically weighted modular forms for $H$ is equal to $M(H)$ and one then obtains Theorem \ref{free}. \\

 \nin A complete determination of the finitely many subgroups $H$ which satisfy the hypothesis of Theorem \ref{free} is given by Bannai, Koike, Munemsasa, and Sekiguchi in \cite{polynomial}. Two such subgroups are $\Gamma$ and $\Gamma_{0}(2).$ The problem of computing an explicit basis for $M(\rho)$ as a $M(H)$-module was investigated in the case when $H = \Gamma$ in the works of Franc and Mason \cite{franc2013fourier}, \cite{survey}, \cite{3dim}. The first part of this paper addresses the analogous problem when $H = \Gamma_{0}(2).$ Solving this problem will  enable us to study the arithmetic of all vector-valued modular forms for $\rho$ by studying the arithmetic of those vector-valued modular forms needed to form a basis for $M(\rho).$ We emphasize that the existence of a weight two modular form on $\Gamma_{0}(2)$ presents an interesting complication when computing a basis for $M(\rho)$ that does not appear if one studies vector-valued modular forms for representations of $\Gamma$.  \\
 
 \nin It will be useful to write down $\C$-algebra generators for $M(\Gamma_{0}(2)).$   
Let $q = e^{2 \pi i \tau}.$ We recall the following functions:
\begin{equation} \label{eqn: E2} E_2(\tau) = 1 - 24 \sum_{n=1}^{\infty} \sigma(n) q^{n} \end{equation}
\begin{equation}
\label{eqn:E4}
 E_4(\tau)  = 1 + 240\sum_{n=1}^{\infty} \sigma_{3}(n)q^{n} 
\end{equation}
\begin{equation} 
\label{eqn:G}
G (\tau)  := -E_2(\tau) + 2E_2(2 \tau)
\end{equation} 
The function $G$ is a weight two modular form on $\Gamma_{0}(2).$ In fact, $M(\Gamma_{0}(2)) = \C[G, E_4].$ In section \ref{derivative}, we show how to use a differential operator to compute a $M(\Gamma_{0}(2))$-basis for $M(\rho)$ if $\rho$ is irreducible and $\textrm{dim } \rho = 2.$ We prove the following:

\begin{restatable}{thm}{basis} \label{thm: basis} Let $\rho: \Gamma_0(2) \rightarrow \textrm{GL}_{2}(\C)$ be an irreducible representation. 
Let $k_0$ denote the integer for which $M_{k_0}(\rho) \neq 0$ and $M_{k}(\rho) = 0$ if $k < k_0.$ 
Let $F$ denote a nonzero element in $M_{k_0}(\rho)$.  
Then $F$ and $D_{k_0}F := \frac{1}{2 \pi i}\frac{dF}{d \tau} - \frac{k_0}{12}E_2 F$ form a basis for $M(\rho)$ as a $M(\Gamma_0(2))$-module.  
\end{restatable}

\nin In particular, $M_{k_0}(\rho) = \C F.$ Thus $F$ is determined by $\rho$ up to multiplication by a nonzero complex number. 
If $k \in \Z$ and if $X \in M_{k} (\rho)$ then $D_{k} X := \frac{1}{2 \pi i}\frac{dX}{d \tau} - \frac{k}{12}E_2 X \in M_{k+2}(\rho).$
Thus $D_{k_0 +2}(D_{k_0}F) \in M_{k_0 +4}(\rho).$ Theorem \ref{thm: basis} implies that there exist unique modular forms $C_1 \in M_{2}(\Gamma_{0}(2))$ and $C_2 \in M_{4}(\Gamma_{0}(2)) $ such that 
\begin{equation}
\label{eqn:C1C2}
D_{k_0 +2}(D_{k_0} F) = C_1 D_{k_0} F + C_2 F.
\end{equation}
In section $4$, we transform the differential equation (\ref{eqn:C1C2}) into a Riemann differential equation on the projective line minus three points by 
locally writing $F$ as a function of the Hauptmodul $\mathfrak{J}$ of $\Gamma_{0}(2),$ where 
\begin{equation}
\label{eqn:J}
 \mathfrak{J} :=\frac{3G^2}{E_4 - G^2}.
\end{equation}

\nin We then give a basis of solutions to the resulting Riemann differential equation on the projective line minus three points in terms of the Gaussian hypergeometric series $_{2}F_{1}.$ We recall that 
\begin{equation}
\label{eqn:Gauss2F1}
 _{2}F_{1}(\alpha, \beta, \gamma; z) = 1 + \sum_{n \geq 1} \frac{(\alpha)_{n} (\beta)_{n}}{  (\gamma)_{n}} \cdot \frac{z^{n}}{n!}, \textrm{where } (\alpha)_{n} := \prod_{i=0}^{n-1} (\alpha + i).
\end{equation}
We also recall the Dedekind $\eta$-function, and the generators $T$ and $S$ of $\Gamma$:  
\begin{equation}
\label{eqn:eta}
\eta = q^{\frac{1}{24}} \prod_{n=1}^{\infty} (1 - q^n).
\end{equation}
\begin{equation}
\label{eqn:T}
 T =  \left[ {\begin{array}{cc}
   1 & 1 \\
   0 & 1 \\
\end{array} } \right]
\end{equation}
\begin{equation}
\label{eqn:S}
S =  \left[ {\begin{array}{cc}
   0 & 1 \\
   -1 & 0 \\
\end{array} } \right]
\end{equation}
\nin We prove the following theorem in Section \ref{hypergeometric}, which gives explicit formulas for the component functions of $F.$

\begin{restatable}{thm}{explicitbasis}\label{thm: explicitbasis}  Let $\rho$ denote an irreducible complex representation of $\Gamma_{0}(2)$ of dimension two such that $\rho(T)$ is diagonalizable. Let $k_0$ denote the least integer for which $M_{k_0}(\rho) \neq 0$ and let $F$ denote a nonzero element in $M_{k_0}(\rho)$. 
  There exist complex numbers $A,B,$ and $r$, which are determined by $\rho$, together with a matrix $Q \in \textrm{GL}_{2}(\C)$ such that 
  $$F(\tau) = Q \left[ \begin{matrix}\eta^{2k_{0}}(\tau) (\mathfrak{J}(\tau)-1)^{r}\mathfrak{J}(\tau)^{-A}\;_{2}F_{1}(A, \frac{1}{2}+ A , 1 + A - B; \mathfrak{J}(\tau)^{-1})   \\  \\
   \eta^{2k_{0}}(\tau) (\mathfrak{J}(\tau)-1)^{r}\mathfrak{J}(\tau)^{-B}\;_{2}F_{1}(B, \frac{1}{2} + B, 1 + B - A; \mathfrak{J}(\tau)^{-1}) \end{matrix} \right]. $$
\end{restatable}

\nin The way that $A,B,$ and $r$ are determined by $\rho$ is explained in detail in section $4$. We shall also show in section four that the hypotheses on $\rho$ in \cref{thm: explicitbasis} imply that $A - B \not \in \Z$ and thus that $A \neq B.$ \\

\nin The Fourier coefficients of $F$ need not be algebraic numbers. We resolve this problem by defining a vector-valued function $F'$, which is obtained by scaling the component functions of $F$ so that their leading Fourier coefficients equal one. Under suitable conditions on $\rho$, the Fourier coefficients of $F'$ will be algebraic numbers. The function $F'$ is a vector-valued modular form for a representation, which we denote by $\rho'.$ We explain in section $5$ that $\rho'$ is conjugate to $\rho$. We prove the following theorem concerning $F'$ and all vector-valued modular forms for $\rho':$ 

\begin{restatable}{thm}{algebraicbasis} \label{thm: algebraicbasis} If $\rho(T)$ has finite order and if $r \in \overline{\Q}$ then all of the Fourier coefficients of both of the component functions of $F'$ are elements of 
$\Q(r)$ and are therefore algebraic numbers. Moreover, for each $k \in \Z$, there exists a basis of $M_{k}(\rho')$ consisting of vector-valued modular forms whose component functions have Fourier coefficients which are elements of $\Q(r)$ and thus are algebraic numbers. 
\end{restatable}

\nin  We are most interested in studying the arithmetic of vector-valued modular forms for $\rho$ when the modular forms $C_1 \in M_{2}(\Gamma_{0}(2)) =\C G$ and $C_2 \in M_{4}(\Gamma_{0}(4)) = \C G^{2} \bigoplus \C E_4$ in equation (\ref{eqn:C1C2}) have rational Fourier coefficients. Let $a,b,c \in \C$ denote the unique complex numbers such that 
\begin{equation}
\label{eqn:abc}
C_1 = -aG \textrm{ and } C_2 = -bG^2 - cE_4.
\end{equation}
\nin We note that $C_1, C_2 \in \Q[[q]]$ if and only if $a,b,c \in \Q.$ If $\rho(T)$ has finite order then we explain in Section \ref{arithmetic} that $a,b,c \in \Q$ if and only if $c \in \Q.$ In this case, $[\Q(r): \Q] \leq 2.$ In this paper, our focus is the case when $[\Q(r): \Q] = 2.$ \\

\nin Let $I$ denote the identity matrix and let 
\begin{equation}
\label{eqn:U}
 U =  \left[ {\begin{array}{cc}
   1 & 0 \\
   1 & 1 \\
\end{array} } \right]. 
\end{equation}

\nin It is well-known that $\Gamma(2) = \langle T^2$ , $U^{2}, -I \rangle$ and 
$\Gamma_{0}(2) = \langle  T, U^2, -I \rangle.$ The group $\Gamma(2)/\langle -I \rangle$ is freely generated by 
$T^2 \langle -I \rangle$ and $U^2 \langle -I \rangle.$ Thus for any $\xi_1, \xi_2 \in \C,$ there exists a character $\psi$ of $\Gamma(2)$ for which $e^{2 \pi i \xi_1} = \psi(T^2), e^{2 \pi i \xi_2}= \psi(U^2),$ and $1 = \psi(-I).$ We now state the main result of the paper. 

\begin{restatable}{thm}{inducedmain} \label{thm: inducedmain}
 Let $\psi$ denote a character of $\Gamma(2)$ such that $\psi(-I) = 1$ and such that there exists some $\xi_1 \in \Q$ and some algebraic number $\xi_{2}$ of degree two for which $e^{2 \pi i \xi_{1}} = \psi(T^2)$ and $e^{2 \pi i \xi_{2}} = \psi(U^2).$ Let $\rho = \textrm{Ind}_{\Gamma(2)}^{\Gamma_{0}(2)} \psi$. Let $k \in \Z$ and $Z \in M_{k}(\rho')$ whose component functions $Z_1$ and $Z_2$ have the property that all of their Fourier coefficients are algebraic numbers. Then the sequence of the denominators of the Fourier coefficients of $Z_1$ and the sequence of the denominators of the Fourier coefficients of $Z_2$ are unbounded. \\

\nin Let $M$ denote the square-free integer such that $\Q(\sqrt{M}) = \Q(\xi_2).$ If $p$ is any sufficiently large prime such that $M$ is not a quadratic residue mod $p$ then $p$ divides the denominator of at least one Fourier coefficient of $Z_1$ and $p$ divides the denominator of at least one Fourier coefficient of $Z_2.$ \end{restatable}
\begin{remark} We explain in \cref{thm: examples} that the hypothesis of \cref{thm: inducedmain} implies that the
 field $\Q(r)$, which is determined by $\rho$, is a quadratic field and that $\Q(r) = \Q(\xi_2).$ It follows from \cref{thm: algebraicbasis} that the hypothesis of \cref{thm: inducedmain} implies that for any $k \in \Z$, there exists a basis of $M_{k}(\rho')$ whose Fourier coefficients are elements of $\Q(\xi_2).$ \end{remark}

\section{Acknowledgements} 

\nin I thank Geoff Mason for teaching me about vector-valued modular forms and for many inspiring discussions.
I thank Cameron Franc for his eagerness to discuss vector-valued modular forms and for his encouragement. I thank Noriko Yui for the invitation to discuss this work in her seminar and for her interest in this work.  I thank the anonymous referee for many helpful suggestions. One of the referee's questions led me to formulate and prove \cref{thm: examples} and to give a more explicit version of \cref{thm: inducedmain}.

\section{The modular derivative and a basis for vector-valued modular forms} \label{derivative}
\subsection{Preliminaries on vector-valued modular forms} 

\nin We now describe what it means for a vector-valued modular form $F$ of integral weight $k$ to be holomorphic at a cusp of  $H$\textbackslash$(\mathfrak{H} \bigcup \mathbb{P}^{1}(\Q))$. Our exposition closely follows \cite{survey}. As $H$ is a finite index subgroup of $\Gamma$,  the subgroup $\bigcap \limits_{\gamma \in \Gamma} \gamma^{-1}H \gamma$ is a finite index subgroup of $\Gamma.$  Therefore there exists 
a smallest positive integer $N$ for which $T^{N} \in \bigcap \limits_{\gamma \in \Gamma} \gamma^{-1}H \gamma.$ We now fix some $\gamma \in \Gamma$ and we will explain what it means for $F$ to be holomorphic at the cusp $\gamma \cdot \infty$. Let  $h \in H$ such that $\gamma T^{N} = h \gamma$. 
Then $$(F|_{k} \gamma)|_{k} T^{N} = F|_{k}(\gamma T^{N}) = F|_{k}(h \gamma) = (F|_{k}h)|_{k} \gamma  = (\rho(h)F)|_{k} \gamma = \rho(h) (F|_{k}\gamma).$$ Let $A$ be an invertible matrix such that 
$A \rho(h) A^{-1}$ is in \textit{modified} Jordan canonical form. A matrix is in \textit{modified} Jordan canonical form if it is a block diagonal matrix whose blocks are of the form 
$$\begin{bmatrix} \lambda &            &             &                 \\
                             \lambda & \ddots &             &                 \\
                                 &          \ddots & \ddots   &                  \\
                             &            &                         \lambda  & \lambda  \end{bmatrix}.$$ 
\nin We let the above matrix denote a block for the modified Jordan canonical form of $A \rho(h) A^{-1}.$
The number $\lambda$ is an eigenvalue of $\rho(h)$. We note that $$((AF)|_{k} \gamma)|_{k} T^{N} = A(F|_{k} \gamma T^{N}) = A \rho(h) (F|_{k} \gamma) = A \rho(h) A^{-1}(AF)|_{k} \gamma.$$
Let $\mu$ denote a complex number for which $\lambda = e^{2 \pi i \mu}$ and let $q_N = e^{\frac{2 \pi i \tau}{N}}.$  Mason and Knopp have proven (Theorem 2.2 in \cite{logvvmf}) that there exist functions $h_i$ such that the component functions of $(AF)|_{k} \gamma$ corresponding to the block above (whose row and column size we denote by $m$) have the form 
$$\begin{bmatrix} h_1 \\
h_2 + \tau h_1 \\
h_3 + \tau h_2 + \begin{binom}{\tau}{2} \end{binom} h_1 \\
\vdots  \\
h_m + \tau h_{m-1} + \begin{binom}{\tau}{2} h_{m-2}\end{binom} + \cdots + \begin{binom}{\tau}{m-1}\end{binom}h_1
\end{bmatrix}$$ where each $h_i$ has a Fourier expansion of the form $h_i(\tau) = \sum_{n \in \Z} a_{n}(i) q_{N}^{n + \mu}.$ \\

\nin We say that $F$ (or equivalently $AF$) is \textbf{meromorphic, holomorphic, or cuspidal} at the cusp $\gamma \cdot \infty$ if the $q_N$-expansion of each $h_i$ has respectively only finitely many nonzero coefficients $a_n$ 
for which $\textrm{Re}(n + \mu) < 0$, no nonzero coefficients $a_n$ for which  $\textrm{Re}(n + \mu) < 0$, and no nonzero coefficients $a_n$ for which $\textrm{Re}(n + \mu) \leq 0.$
 We note that this definition is independent of the choice of $\mu.$ If at least one of the component functions of $AF$ contains a term with a nonzero power of $\tau$ then we say that $AF$ and $F$ are \textbf{logarithmic vector-valued modular forms}. The focus of this paper is the study of vector-valued modular forms which are not logarithmic. We observe that if $F$ has a holomorphic $q$-expansion at the cusp $\infty$ which is not logarithmic then $\rho(T^{N})$ is diagonalizable. If $H = \Gamma_{0}(2)$ then this condition is equivalent to requiring that $\rho(T)$ is diagonalizable. \textit{We therefore only consider representations $\rho$ of $\Gamma_{0}(2)$ for which $\rho(T)$ is diagonalizable.} 
\subsection{An explicit basis}

\nin Theorem \ref{free} states that if $M(H)$ is a polynomial ring in two variables then $M(\rho)$ is a free $M(H)$-module whose rank equals the dimension of $\rho.$   We recall that $E_2(\tau): = 1 - 24 \sum_{n=1}^{\infty} \sigma(n) q^{n}$, $G(\tau) := -E_2(\tau) + 2E_2(2 \tau)$, and  that $M(\Gamma_{0}(2)) = \C[G, E_4].$ The purpose of this section to use the modular derivative $D_{k},$ which we define below, to describe a basis for $M(\rho)$ as a $M(\Gamma_{0}(2))$-module when $\rho$ is a two-dimensional irreducible representation of $\Gamma_{0}(2)$. \\

\nin  Let $k \in \Z.$  If $A$ is a holomorphic or meromorphic function from $\mathfrak{H}$ to $\C^{n}$ then: 
\begin{equation}
\label{Dk} D_{k} A := \frac{1}{2 \pi i}\frac{dA}{d \tau} - \frac{k}{12}E_2 A = q \frac{d}{dq} A - \frac{k}{12}E_2 A.
\end{equation} 

\nin  The modular derivative $D_{k}$ has the lovely property (section $10.5$ in \cite{lang}) that for all $\gamma \in \Gamma,$ $D_{k}(A|_{k} \gamma) = (D_{k}A)|_{k+2} \gamma.$ This property implies that if $F \in M_{k}(\rho)$ then
$D_{k} F \in M_{k+2}(\rho)$ and if $m \in M_{k}(H)$ then $D_{k} m \in M_{k+2}(H).$ The linear maps $D_{k}: M_{k}(\rho) \rightarrow M_{k+2}(\rho)$  and $D_{k}: M_{k}(H) \rightarrow M_{k+2}(H)$ are quite useful. We shall make use of the following notation:
\begin{equation}
\label{eqn:theta} \theta := D_{0} = q \frac{d}{dq} = \frac{1}{2 \pi i}\frac{d}{d \tau}
\end{equation} 

\nin The goal of this section is to prove \cref{thm: basis}. 

\basis*

\nin We note that the fact that $M(\rho)$ is a free $M(\Gamma_{0}(2))$-module implies that if $k$ is sufficiently negative then $M_{k}(\rho) = 0.$
We will use the following two results to prove \cref{thm: basis}.

\begin{lemma} \label{irred} Let $\rho$ be an irreducible representation of a finite index subgroup $H$ of $\Gamma$.
Let $F$ be a nonzero vector-valued modular form of weight $k \in \Z$ with respect to $\rho$. Then the component functions of $F$ are linearly independent over $\C$. 
\end{lemma}

\begin{proof} Let $n$ denote the dimension of $\rho$, let $f_1,...,f_n$ denote the component functions of $F$, and let $E$ denote the $\C$-span of $f_1,...f_n.$ We view $E$ as a right $H$-module via the action: $g \cdot f_i := f_i|_{k} g.$
The fact that $E$ is a $H$-module is immediate from the fact that $F$ is a vector-valued modular form. 
Let $W$ denote the right $H$-module that furnishes $\rho.$ This means that $(w \cdot \gamma_{1}) \cdot \gamma_{2} = w \cdot (\gamma_1 \gamma_2)$ for all $w \in W$ and $\gamma_{1}, \gamma_{2} \in H$ and that there exists a $\C$-basis $e_1,...,e_n$ of $W$ such that for every $i$,  $e_i \cdot \gamma = \sum_{j =1}^{n} \rho(\gamma)_{i,j} e_{j}.$
We define a map $\psi: W \rightarrow E$ by setting $\psi(e_i) = f_i$ and extending linearly. 
We now check that the map $\psi$ is a map of $H$-modules. Let $g \in H$ and let $g_{i,j}$ denote the $i$-th row and $j$-th column entry of $\rho(g).$ We have that 
$\psi(e_i) \cdot g =  f_i  \cdot g = f_i|_{k} g = \sum_{j=1}^{n}\rho(g)_{i,j} f_{j} = \sum_{j=1}^{n} g_{i,j}  \psi(e_i) = \psi(\sum_{j=1}^{n} g_{i,j} e_i) = \psi(e_i \cdot g).$
As $\psi$ is a $H$-module map, $\textrm{ker } \psi$ is a $H$-submodule of $W$. 
As $\rho$ is irreducible, $\textrm{ker } \psi$ is equal to either $0$ or $W.$
As each $f_i = \psi(e_i)$ and $F \neq 0$, we have that $E \neq 0.$ Thus $\textrm{ker } \psi \neq W$ and so $\psi$ is injective. It is clear that $\psi$ is surjective and thus $\psi$ is an isomorphism. Hence the elements $f_1,...,f_n$ are linearly independent over $\C.$
\end{proof}

\begin{lemma} \label{parity} If $\rho$ is an irreducible representation of a finite index subgroup $H$ of $\Gamma$ for which $-I \in H$ then there exists an integer $k$ such that $\rho(-I)  = (-1)^{k}I$ and the weights of all the homogeneous elements in $M(\rho)$ are congruent to $k$ modulo two. 
\end{lemma}

\begin{proof}  Let $W$ denote the right $H$-module that furnishes $\rho.$ As $\rho(-I)^{2} = 1$, the eigenvalues of $\rho(-I)$ are $1$ and $-1.$ Let $W^{1}$ denote the $+1$-eigenspace and $W^{-1}$ denote the $-1$-eigenspace. We note that $\rho(-I)$ is in the center of $\textrm{Im }\rho$ since $-I$ is in the center of $H.$ Hence $W^{1}$ and $W^{-1}$ are right $H$-submodules of $W.$  As $\rho$ is irreducible, every right $H$-submodule of $W$ is either $W$ or $0.$ Hence there exists some $k \in \Z$ such that $W^{(-1)^{k}} = W$ and $W^{(-1)^{k+1}} = 0.$ Thus $\rho(-I) = (-1)^{k} I.$ Finally, if $B$ is a nonzero vector-valued modular form of weight $j$ then $(-1)^{j}B = B|_{j}-I = \rho(-I)B = (-1)^{k} B$ and thus $j$ and $k$ have the same parity. 
\end{proof}

\nin We now give the proof of  \cref{thm: basis}. 

\begin{proof} Let $F_1, F_2$ denote a basis of vector-valued modular forms for $M(\rho)$. 
The crux of our proof is to show that the weights of $F_1$ and $F_2$ are not equal.
We proceed by contradiction and suppose that the weights of $F_1$ and $F_2$ are equal.  Then there exist $a,b,c,d \in \C$ such that 
$D_{k_0}F_1 = aGF_1 + bGF_2$ and $D_{k_0} F_2 = cGF_1 + dGF_2.$ 
We rewrite this pair of equations as follows: $$ \left[ \begin{array}{c}
D_{k_{0}}F_1\\
D_{k_{0}}F_2\\
\end{array}\right] =  \begin{bmatrix} a & b\\   c & d \\ \end{bmatrix} G \left[ \begin{array}{c}
F_1\\
F_2\\
\end{array}\right].  $$
If $P$ is a $2 \times 2$ matrix then we note that $$P \left(G  \left[ \begin{array}{c}
F_1\\
F_2\\ \end{array}\right] \right)   = G \left(P  \left[ \begin{array}{c}
F_1\\
F_2\\ \end{array}\right] \right)$$
Thus if $P$ is a $2 \times 2$ invertible matrix then we have that \begin{align*} P \left[ \begin{array}{c}
D_{k_{0}}F_1\\
D_{k_{0}}F_2\\ \end{array} \right] & =  P \begin{bmatrix} a & b\\   c & d \\ \end{bmatrix} G \left[ \begin{array}{c}
F_1\\
F_2\\ \end{array}\right] \\ 
& =  P \begin{bmatrix} a & b\\   c & d \\ \end{bmatrix} P^{-1} \left(P \left( G \left[ \begin{array}{c}
F_1\\
F_2\\ \end{array} \right] \right) \right) \\
& =  P \begin{bmatrix} a & b\\   c & d \\ \end{bmatrix} P^{-1} \left(G \left( P \left[ \begin{array}{c}
F_1\\
F_2\\ \end{array} \right] \right) \right). \end{align*}
We may put the matrix $\begin{bmatrix} a & b\\   c & d \\ \end{bmatrix}$ in Jordan canonical form and we now choose $P$ so that  $P\begin{bmatrix} a & b\\   c & d \\  \end{bmatrix} P^{-1} = \begin{bmatrix} * & * \\   0 & \lambda \\ \end{bmatrix}$ for some $\lambda \in \C$.
 We define the functions $A_1$ and $A_2$ by $$ \left[ \begin{array}{c} A_1 \\ A_2 \\
\end{array}\right] = P \left[ \begin{array}{c}F_1\\ F_2\\ \end{array}\right].$$
As $P$ is invertible, the functions $A_1$ and $A_2$ are a basis for $M(\rho).$ We now have that \begin{align*} \left[ \begin{array}{c} D_{k_0} A_1 \\ D_{k_0} A_2 \\
\end{array}\right]  & = P \left[ \begin{array}{c}D_{k_0} F_1\\ D_{k_0} F_2\\ \end{array}\right]\\  & =  P \begin{bmatrix} a & b\\   c & d \\ \end{bmatrix} P^{-1} \left( G \left( P \left[ \begin{array}{c}
F_1\\
F_2\\ \end{array} \right ] \right) \right) \\ 
& = \begin{bmatrix} * & * \\   0 & \lambda \\ \end{bmatrix} G \left[ \begin{array}{c}
A_1\\
A_2\\ \end{array} \right] \\
 &  =   \left[ \begin{array}{c}
* \\ \lambda G A_2 \\ \end{array} \right].   \end{align*}
\nin Thus $D_{k_0} A_2 = \lambda G A_2$. Therefore the two component functions of the vector-valued function $A_2$ satisfy an ordinary differential equation of order one and must be linearly dependent.  As $A_2$ is part of a basis for $M(\rho)$, $A_2 \neq 0.$ Lemma \ref{irred} states that the component functions of any nonzero vector-valued modular form with respect to an irreducible representation are linearly independent. We have thus shown that the components of $A_2$ are both linearly dependent and independent, a contradiction. We conclude that the weights of $F_1$ and $F_2$ are not equal. \\

\nin We recall that $F \in M_{k_0}(\rho)$ such that $F \neq 0$. 
We have shown that the weights of a $M(\Gamma_{0}(2))$-basis for $M(\rho)$ cannot be equal and hence that 
$M_{k_0}(\rho) = \C F$. We may therefore take $F$ to be an element of a basis for $M(\rho)$. Let $B$ denote a homogeneous element in $M(\rho)$ such that $F$ and $B$ form a basis for $M(\rho)$. 
We claim that the weight of $B$, which we denote by $w$, is equal to $k_0 +2$. It follows from Lemma \ref{parity} that $M_{k_0+1}(\rho) = 0.$ Thus $w \geq k_0 + 2.$
If $w >k_0 +2$ then $D_{k_0} F =  mF$ for some $m \in M_{2}(\Gamma_{0}(2))$. 
Thus if $w > k_0 + 2$ then $\frac{1}{2 \pi i} \frac{dF}{d \tau} - \frac{k_0}{12}E_2 F = mF$ and the two component functions of $F$ would satisfy this ordinary differential equation of order one and therefore be linearly dependent. This would contradict Lemma \ref{irred} as $F \neq 0$ and $\rho$ is irreducible. Thus the weight of $B$ is $k_0 + 2$. 
We then have that $D_{k_0}F = \alpha F + \gamma B$ where $\alpha \in M_{2}(\Gamma_{0}(2))$ and $\gamma \in M_{0}(\Gamma_0(2)) = \C$.  If $\gamma = 0$ then $D_{k_0} F = \alpha F$ and so the two component functions of $F$ are linearly dependent, which would contradict Lemma \ref{irred}. Hence $\gamma \neq 0$. 
Thus $B \in \textrm{span}_{M(\Gamma_0(2))}(F, D_{k_0}F)$. As $M(\rho)$ is spanned by $F$ and $B$, it is also spanned by $F$ and $D_{k_0}F$. Finally, as $M(\rho)$ is a free module of rank two over $M(\Gamma_0(2))$, an integral domain, and $F$ and $D_{k_0}F$ span $M(\rho)$, we conclude that $F$ and $D_{k_0}F$ form a basis for $M(\rho)$. \end{proof}

\section{Hypergeometric Series} \label{hypergeometric}

\nin From this point forwards, $\rho$ will denote a complex \textbf{irreducible} representation of $\Gamma_{0}(2)$ of dimension two and $k_0$ will denote the least integer for which $M_{k_0}(\rho) \neq 0$.  We let $F$ denote a nonzero element in  $M_{k_{0}}(\rho).$ We proved in  \cref{thm: basis} that $F$ and $D_{k_0}F$ form a basis for $M(\rho)$ as a $M(\Gamma_{0}(2))$-module. In particular, $M_{k_0}(\rho) = \C F$. Hence $F$ is determined by $\rho$ up to multiplication by a nonzero complex number. In this section, we use  \cref{thm: basis} to compute an ordinary differential equation that $F$ satisfies. We will then solve this differential equation explicitly using the Dedekind $\eta$-function, the Gaussian hypergeometric series $_{2}F_{1}$,  and a Hauptmodul of $\Gamma_{0}(2)$. \\

\nin  As $F \in D_{k_0}(\rho)$, $D_{k_0+2}(D_{k_0}(F)) \in M_{k_0 + 4}(\rho).$
Thus \cref{thm: basis} implies that $D_{k_0 +2}(D_{k_0} F) = C_1 (D_{k_0}F) + C_2 F$ where $C_1 \in M_{2}(\Gamma_{0}(2)) = \C G$ and $C_2 \in M_{4}(\Gamma_{0}(2)) = \C G^2 \bigoplus \C E_4$. 
 Let $a,b,c \in \C$ be the complex numbers such that $C_1 = -aG$ and $C_2 = -(bG^{2} + cE_4)$. Hence \begin{equation} \label{eqn:MLDE} D_{k_0 +2}(D_{k_0}F) + aG D_{k_{0}}F + (bG^2  + cE_4)F = 0. \end{equation} 

\nin  We make use of the Dedekind $\eta$-function to solve the differential equation (\ref{eqn:MLDE}). The function $\eta^{2}$ is holomorphic in $\mathfrak{H}$ and it does not vanish in $\mathfrak{H}$. Let $\omega$ denote the character of $\Gamma$ for which $\eta^{2}|_{1} g  = \omega(g) \eta^2.$ 
Let $F_0 := \frac{F}{\eta^{2k_{0}}}$.  We observe that for all $g \in \Gamma_{0}(2)$, $$F_0|_{0} g = (\eta^{-2k_0}|_{k_0} g)(F|_{k_0} g) = \omega^{-k_0}(g) \eta^{-2k_0} \rho(g) F = 
(\rho \otimes \omega^{-k_0})(g) \eta^{-2k_0} F = (\rho \otimes \omega^{-k_0})(g) F_{0}.$$
\nin This observation together with the fact that $\eta^{2}$ is a holomorphic non-vanishing function of $\mathfrak{H}$ implies that  $F_0$ is a meromorphic vector-valued modular form of weight zero with respect to the representation $\rho_{0} := \rho \otimes \omega^{-k_0}.$ We note that $F_{0}$ is holomorphic in $\mathfrak{H}$ and the only possible poles of $F_{0}$ occur at the cusps. 
We now compute a differential equation that $F_0$ satisfies. We sometimes use the notation $D_{k}^{2}$ to denote $D_{k+2} \circ D_{k}.$ We recall that $\theta := D_{0} = q \frac{d}{dq} = \frac{1}{2 \pi i} \frac{d}{d \tau}.$

\begin{lemma}   Let $F_0 := \frac{F}{\eta^{2k_{0}}}$.  Then \begin{equation} \label{eqn:MLDE0} D_{0}^{2} (F_{0})+ aGD_{0}(F_{0}) + (bG^2 + cE_4)F_{0} = 0. \end{equation} 

\end{lemma}
\begin{proof} Let $f$ denote a function on $\mathfrak{H}$, let $k \in \Z,$ and let $g = \frac{f}{\eta^{2k}}$. To prove the lemma, we observe that because $\eta$ never vanishes in $\mathfrak{H}$, it suffices to show that$$\eta^{2k}(D_0^2(g) + aGD_{0}(g) + (bG^2 + cE_4)g) = D_k^2(f) + aGD_{k}(f) + (bG^2 + cE_4)f .$$ To show that the equation above holds, it suffices to prove that $D_k(f) = D_{0}(\frac{f}{\eta^{2k}}) \eta^{2k}$ and that $D_k^{2}(f) =  \eta^{2k} D_{0}^{2}(\frac{f}{\eta^{2k}}).$ 
We recall that $\theta(\eta) = \frac{1}{2 \pi i} \eta^{'} = \frac{1}{24} E_2 \eta$ (Section 5.8 in \cite{cohen}). Thus $\theta(\eta^{2k}) = 2k \eta^{2k-1}  \theta(\eta) = \frac{k E_2 \eta^{2k}}{12}.$ Hence $D_{k}(\eta^{2k}) = \theta(\eta^{2k}) - \frac{k}{12}E_2 \eta^{2k} = 0.$ If $t, l \in \Z$ and if $\alpha, \beta$ are holomorphic functions 
then $D_{t+l}(\alpha \beta) = \alpha D_{t} \beta + \beta D_{l} \alpha.$
\nin We now have that
 \begin{align*} D_k(f) & = D_{k} \left(\eta^{2k} \cdot \frac{f}{\eta^{2k}} \right) \\
					  & = \frac{f}{\eta^{2k}}D_{k}(\eta^{2k}) + \eta^{2k} \theta \left(\frac{f}{\eta^{2k}}\right) \\
					  & = \eta^{2k} D_{0} \left(\frac{f}{\eta^{2k}} \right). \end{align*} 
\nin Thus \begin{align*} D_k^{2}(f) & = D_{k+2}(D_{k}(f)) \\
									& = D_{k+2}\left(\eta^{2k} \theta \left(\frac{f}{\eta^{2k}} \right) \right) \\
									& = \theta \left(\frac{f}{\eta^{2k}} \right) D_{k}(\eta^{2k}) + \eta^{2k} D_{2} \left(\theta \left(\frac{f}{\eta^{2k}}\right)\right) \\
									& = \eta^{2k} D_{2} \left(\theta \left(\frac{f}{\eta^{2k}} \right) \right) \\
									& = \eta^{2k} D_{0}^{2} \left(\frac{f}{\eta^{2k}}\right).
									\end{align*}
								
\end{proof} 

\nin We now explain how to make a change of variables to solve the differential equation (\ref{eqn:MLDE0}). This technique can be found in a paper of Kaneko and Zagier \cite{kanekozagier}. The modular curve 
\newline \nin  $\Gamma_{0}(2) \backslash (\mathfrak{H} \bigcup \mathbb{P}^{1}(\Q))$ is a compact Riemann surface of genus zero.  In a fundamental region, its cusps are $0$ and $\infty$ and its elliptic point is $\frac{1 + i}{2}$. We shall make use of the function $\mathfrak{J} := \frac{3G^{2}}{E_4 - G^2}$ to solve the differential equation (\ref{eqn:MLDE0}). The function $\mathfrak{J}$ is a modular function and it only has one pole, which is a simple pole at $\infty.$  Hence $\mathfrak{J}$ induces a complex-analytic isomorphism  from $\Gamma_{0}(2) \backslash (\mathfrak{H} \bigcup \mathbb{P}^{1}(\Q))$ to $\mathbb{P}^{1}(\C).$ The function $\mathfrak{J}$ is a therefore a Hauptmodul. \\

\nin We will solve the differential equation $(4)$ that $F_0$ satisfies by writing $F_0$ locally as a function of $\mathfrak{J}$. The function $\mathfrak{J}$ enjoys the property that $\mathfrak{J}(\tau_1) = \mathfrak{J}(\tau_2)$ if and only if $\Gamma_{0}(2) \cdot \tau_1 = \Gamma_{0}(2) \cdot \tau_{2}$. This fact implies that if $\tau \in \mathfrak{H}$ such that $\tau$ is not an elliptic point of $\Gamma_{0}(2)$ then there exists a connected and  simply connected open set $U_{\tau}$ containing $\tau$ for which $U_{\tau}  \subset \mathfrak{H}$  and the restriction of $\mathfrak{J}$ to $U_{\tau}$ is injective. In particular, $U_{\tau}$ does not contain an elliptic point. Consequently, if $\tau \in \mathfrak{H}$ which is not an elliptic point then there exists a unique function $H$ such that $F_{0}|_{U_{\tau}} = H \circ \mathfrak{J}.$ We note that $H$ is holomorphic since $\mathfrak{J}|_{U_{\tau}}$ is biholomorphic and $F_0$ is  holomorphic.  The next theorem computes the differential equation that $H$ satisfies. 

\begin{thm} \label{diff} Let $\tau_{0} \in \mathfrak{H}$ such that $\tau_{0}$ is not an elliptic point of $\Gamma_{0}(2)$.  Let $U_{\tau_{0}}$ denote a connected and  simply connected open set  containing $\tau_{0}$ for which $U_{\tau_{0}}  \subset \mathfrak{H}$ 
and the restriction of $\mathfrak{J}$ to $U_{\tau_{0}}$ is injective.  Let $F_{0} = \frac{F}{\eta^{2k_{0}}}$ and let $H$ denote the function for which $F_{0}|_{U_{\tau_{0}}} = H \circ \mathfrak{J}.$ We have that for all $\tau \in U_{\tau_{0}}$, 
\begin{equation} \label{eqn:ODE}
H^{''}(\mathfrak{J}(\tau)) + \frac{7\mathfrak{J}(\tau) - 6a\mathfrak{J}(\tau) - 3}{6\mathfrak{J}(\tau)(\mathfrak{J}(\tau) -1)} H^{'}(\mathfrak{J}(\tau)) + \frac{(b+c)\mathfrak{J}(\tau) +3c}{\mathfrak{J}(\tau)(\mathfrak{J}(\tau)-1)^{2}} H(\mathfrak{J}(\tau)) = 0  \end{equation} \end{thm} 

\nin The proof of Theorem \ref{diff} uses the following propositions whose proofs are appear in the appendix.
\begin{proposition}\label{appendix1} $$\theta(\mathfrak{J}) = (1 - \mathfrak{J})G.$$  
\end{proposition}

\begin{proposition} \label{appendix2} $$\theta^{2}(\mathfrak{J}) =G^2(1 - \mathfrak{J})(\frac{3 - 7 \mathfrak{J}}{6\mathfrak{J}})  + \frac{1}{6}E_2 \theta(\mathfrak{J}).$$
\end{proposition} 
\nin We will also need to use the fact that $\frac{E_4}{G^2} = \frac{\mathfrak{J} + 3}{\mathfrak{J}}$ in our proof of Theorem \ref{diff}. This equality is immediate from the fact that $\mathfrak{J}: = \frac{3G^2}{E_4 - G^2}.$
We now proceed with the proof of Theorem \ref{diff}. 
\begin{proof} (\textbf{Proof of Theorem \ref{diff}}) We have that $D_0 F_{0} = \theta F_0$ and $D_0^{2} F_0 = (\theta - \frac{1}{6} E_2)(\theta F_0) = \theta^{2} F_{0} - \frac{1}{6} E_2 \theta F_0.$
Therefore $$D_0^{2}F_0 + aGD_0 F_0 + (bG^2 + cE_4)F_0 = \theta^{2} F_{0} + (aG - \frac{1}{6} E_{2}) \theta F_{0} + (bG^2 + cE_4)F_0 = 0.$$
\nin We have that 
\begin{align*}  \theta(F_{0})(\tau) &  =  \theta(H \circ \mathfrak{J})(\tau) \\ 
& = \frac{1}{2 \pi i} (H \circ \mathfrak{J})^{'}(\tau) \\
& =   H^{'}(\mathfrak{J}(\tau)) \cdot  \frac{1}{2 \pi i} \mathfrak{J}^{'}(\tau) \\
& =   H^{'}(\mathfrak{J}(\tau)) \cdot (\theta \mathfrak{J})(\tau) \\
& =  H^{'}(\mathfrak{J}(\tau)) \cdot G(\tau)(1 - \mathfrak{J}(\tau)).   \end{align*} 
Therefore 
 \begin{align*} & \theta^{2}(F_0)(\tau) \\  
 &=  \theta((H^{'} \circ \mathfrak{J}) \cdot \theta \mathfrak{J})(\tau) \\
& = \theta(H^{'} \circ \mathfrak{J})(\tau) \cdot (\theta \mathfrak{J})(\tau) + (H^{'} \circ \mathfrak{J})(\tau) \cdot (\theta^{2} \mathfrak{J})(\tau) \\
& = H^{''}(\mathfrak{J}(\tau)) \cdot \theta( \mathfrak{J})(\tau) \cdot  \theta(\mathfrak{J})(\tau)   + (H^{'} \circ \mathfrak{J})(\tau) \cdot (\theta^{2} \mathfrak{J})(\tau) \\
& =  H^{''}(\mathfrak{J}(\tau)) \cdot (1 - \mathfrak{J}(\tau))^{2} G^{2}(\tau)  \\
& \qquad + H^{'}( \mathfrak{J}(\tau)) \cdot G^2(\tau)(1 - \mathfrak{J}(\tau))\Bigg(\frac{3 -7\mathfrak{J}(\tau)}{6\mathfrak{J}(\tau)}\Bigg)  \\
& \quad \qquad + \frac{1}{6}E_2(\tau) H^{'}(\mathfrak{J(\tau)}) \theta(\mathfrak{J})(\tau)  \\
&  = G^{2}(\tau)\Bigg(H^{''}(\mathfrak{J}(\tau)) (1 - \mathfrak{J}(\tau))^{2} + (H^{'} (\mathfrak{J}(\tau))) \cdot (1 - \mathfrak{J}(\tau)) \cdot \frac{3 - 7\mathfrak{J}(\tau)}{6 \mathfrak{J}(\tau)}\Bigg) \\
& \quad \qquad + \frac{1}{6} E_2(\tau) \theta(\mathfrak{J})(\tau) H^{'}(\mathfrak{J}(\tau)). \end{align*}  
Thus \begin{align*} & D_0^{2}F_0(\tau) + aG(\tau)D_0 F_0(\tau) + \Big(bG^2(\tau) + cE_4(\tau)\Big)F_0(\tau) \\
&  = \theta^{2} F_{0}(\tau) + \Big(aG(\tau) - \frac{1}{6} E_{2}(\tau)\Big) \theta F_{0}(\tau) + \Big(bG^2(\tau) + cE_4(\tau)\Big)F_0(\tau) \\
& =  G^{2}(\tau)\Bigg(H^{''}(\mathfrak{J}(\tau)) (1 - \mathfrak{J}(\tau))^{2} + H^{'} (\mathfrak{J}(\tau)) \cdot (1 - \mathfrak{J}(\tau))\Big(\frac{3 - 7 \mathfrak{J}(\tau)}{6 \mathfrak{J}(\tau)}\Big)\Bigg) \\ 
&  \quad + \frac{1}{6} E_2(\tau) \theta(\mathfrak{J})(\tau) H^{'}(\mathfrak{J}(\tau)) \\
& \qquad + \Big(aG(\tau) - \frac{1}{6}E_2(\tau)\Big) H^{'}(\mathfrak{J}(\tau)) \cdot G(\tau)(1 - \mathfrak{J}(\tau)) \\
& \qquad \qquad + G^2(\tau) \left(b + c \frac{E_4(\tau)}{G^2(\tau)}\right) H(\mathfrak{J}(\tau))  \\
& = G^{2}(\tau)\Bigg(H^{''}(\mathfrak{J}(\tau)) (1 - \mathfrak{J}(\tau))^{2} + H^{'} (\mathfrak{J}(\tau)) \cdot (1 - \mathfrak{J}(\tau))\Big(\frac{3- 7\mathfrak{J}(\tau)}{6 \mathfrak{J}(\tau)}\Big)\Bigg)  \\
& \qquad + aG(\tau)H^{'}(\mathfrak{J}(\tau)) \cdot G(\tau)(1 - \mathfrak{J}(\tau)) \\
& \qquad \qquad + G^2(\tau)\Bigg(b + c \frac{E_4(\tau)}{G^2(\tau)}\Bigg) H(\mathfrak{J}(\tau)) \\
& = G^2(\tau)H^{''}(\mathfrak{J}(\tau))(1 - \mathfrak{J}(\tau))^{2} \\
& \qquad + G^{2}(\tau)H^{'}(\mathfrak{J}(\tau)) ((1 - \mathfrak{J}(\tau))\Bigg(a + \frac{3 - 7\mathfrak{J}(\tau)}{6\mathfrak{J}(\tau)}\Bigg) \\
& \qquad \qquad + G^2(\tau) \Bigg(b + c \frac{E_4(\tau)}{G^2(\tau)}\Bigg)H(\mathfrak{J}(\tau))\\
& =  G^2(\tau)H^{''}(\mathfrak{J}(\tau))(1 - \mathfrak{J}(\tau))^{2} \\
& \qquad + G^{2}(\tau)H^{'}(\mathfrak{J}(\tau))(1 - \mathfrak{J}(\tau))\left(a  + \frac{3 - 7\mathfrak{J}(\tau)}{6\mathfrak{J}(\tau)}\right) \\
& \qquad \qquad + G^2(\tau) \Bigg(b + c \frac{\mathfrak{J(\tau)} +3}{\mathfrak{J}(\tau)}\Bigg) H(\mathfrak{J}(\tau)). \end{align*}
 \nin We have shown that 
\begin{align*} 0 & = G^{2}(\tau)\Bigg(H^{''}(\mathfrak{J}(\tau))(1 - \mathfrak{J}(\tau))^{2}  + H^{'}(\mathfrak{J}(\tau))(1 - \mathfrak{J}(\tau))\left(\frac{(6a - 7)\mathfrak{J}(\tau) + 3}{6\mathfrak{J}(\tau)}\right)  \\
& \quad +  \left(b + c \frac{\mathfrak{J(\tau)} +3}{\mathfrak{J}(\tau)}\right)H(\mathfrak{J}(\tau))\Bigg). \end{align*} 
The fact that $V := \begin{bmatrix} 1 & -1 \\ 2 & - 1  \end{bmatrix}$ fixes $\frac{1 + i}{2}$  implies that $G^{2}\left(\frac{1 +i}{2}\right) = 0.$ The valence formula then implies that $G^{2}(\tau) = 0$ if and only if $\tau \in \Gamma_{0}(2) \cdot \frac{1 + i}{2}$, which is equivalent to the statement that $\tau$ is an elliptic point. 
Thus if $\tau$ is not an elliptic point then
\begin{align*} 0 & = H^{''}(\mathfrak{J}(\tau))(1 - \mathfrak{J}(\tau))^{2}  + H^{'}(\mathfrak{J}(\tau))(1 - \mathfrak{J}(\tau))\left(\frac{(6a - 7)\mathfrak{J}(\tau) + 3}{6\mathfrak{J}(\tau)}\right)  \\
 \qquad \quad & +  \left(\frac{(b+c)\mathfrak{J}(\tau) +3c}{\mathfrak{J}(\tau)}\right)H(\mathfrak{J}(\tau)). \end{align*}
 \nin As $U_{\tau_{0}}$ does not contain any elliptic points, the above equation holds for all $\tau$ in $U_{\tau_{0}}$. Finally, we divide the above equation by $(1 - \mathfrak{J}(\tau))^{2}$ and we obtain the differential equation (\ref{eqn:ODE}).  \end{proof}

\noindent Let $V_{\tau_{0}} := \mathfrak{J}(U_{\tau_{0}})$ and let $Y := \mathfrak{J}(\tau).$
We have shown that $$D_0^{2}F_0(\tau) + aG(\tau)D_0 F_0(\tau) + \Big(bG^2(\tau) + cE_4(\tau)\Big)F_0(\tau) = 0 \textrm{ for all }\tau \in U_{\tau_{0}}$$ if and only if every $Y \in V_{\tau_{0}}$ satisfies the differential equation \newline  
\begin{equation} \label{eqn:ODEmonic}
H^{''}(Y) + \frac{7Y - 6aY- 3}{6Y(Y -1)} H^{'}(Y) + \frac{(b+c)Y+3c}{Y(Y-1)^{2}} H(Y) = 0
\end{equation}
\nin We will now solve the differential equation (\ref{eqn:ODEmonic}) using hypergeometric series. A good reference for hypergeometric series and singularities of ordinary differential equations is Chapter $6$ of \cite{ODE}.
The singularities of the differential equation (\ref{eqn:ODEmonic}) occur at $Y = 0, 1, \infty$ and they are all regular. \\

\nin The method of Frobenius provides a means to find a power series solution to such a differential equation in the neighborhood of a regular singular point. Let $p(Y) = \frac{7Y - 6aY- 3}{6Y(Y -1)}$ and $q(Y) = \frac{(b+c)Y+3c}{Y(Y-1)^{2}}.$  If $y_0$ is a regular singular point of (\ref{eqn:ODEmonic}) then we define $P_{0} = \lim_{Y \rightarrow y_{0}} (Y - y_0)p(Y)$ and 
$Q_0  = \lim_{Y \rightarrow y_0} (Y - y_0)^{2} q(Y).$ One may write  
$H(Y) = (Y - y_0)^{v}\left(1 + \sum_{n=1}^{\infty} c_n(Y- y_0)^{n}\right)$ in a neighborhood of $y_0$ and the number $v$ satisfies the indicial equation, $v(v-1) + P_0 v + Q_0  = 0.$ The roots of this equation are called the indicial roots at $y_0.$ The indicial equation at $0$ is $Y(Y-1) + \frac{1}{2} Y = 0$ and its roots are $0$ and $\frac{1}{2}.$ The indicial equation at $1$ is $Y(Y-1) + \frac{2 -3a}{3}Y + b + 4c = 0.$ \\

\nin The differential equation (\ref{eqn:ODEmonic}) is Fuchsian since all of its singularities are regular.  In fact, this differential equation is a Riemann differential equation since it is a second order differential equation with exactly three singularities, all of which are regular. We follow the standard technique to solve a Riemann differential equation of order two and define the function $W(Y)$ via the equation $H(Y) = Y^{\lambda} (Y-1)^{r} W(Y)$ where $\lambda$ is an indicial root of (\ref{eqn:ODEmonic}) at $0$ and $r$ is an indicial root of (\ref{eqn:ODEmonic}) at $1$. We make the choice of setting $\lambda = 0$.  We also recall that the indicial roots at $1$ are the roots of the quadratic equation: $$r(r -1) + \left(\frac{2 - 3a}{3}\right)r + (b + 4c) = 0.$$
We have that $ H(Y) = (Y-1)^{r} W(Y).$
\nin Thus  \begin{align*} 0 &= H^{''}(Y) + \frac{7Y - 6aY- 3}{6Y(Y -1)} H^{'}(Y) + \frac{(b+c)Y+3c}{Y(Y-1)^{2}} H(Y) \\
& = r(r-1)(Y-1)^{r-2}W(Y) + 2r(Y-1)^{r-1}W^{'}(Y)  + (Y-1)^{r}W^{''}(Y) \\ 
& \quad +  \frac{7Y - 6aY- 3}{6Y(Y -1)} \left(r(Y-1)^{r-1}W(Y) + (Y-1)^{r}W^{'}(Y)\right) 
\\ & \qquad + \frac{(b+c)Y+3c}{Y(Y-1)^{2}} (Y-1)^{r} W(Y)  \\
 & = (Y-1)^{r} W^{''}(Y) + W^{'}(Y)\left(2r(Y-1)^{r-1} + \frac{7Y - 6aY- 3}{6Y(Y -1)} (Y-1)^{r} \right) \\ 
 &\quad +  W(Y)\left((Y-1)^{r} \frac{(b+c)Y + 3c}{Y(Y-1)^{2}} + r(Y-1)^{r-1} \frac{7Y - 6aY- 3}{6Y(Y -1)} + r(r-1)(Y-1)^{r-2}\right). \end{align*} 
  
\nin In the computation below, we will use the fact that $r$ satisfies the equation $r(r -1) + (\frac{2 - 3a}{3})r + (b + 4c) = 0$  to get that $6b + 6c + 7r -6ar + 6r(r-1) = 6b + 6c + 7r -6ar -6(b+ 4c) - 2r(2 - 3a) = -18c + 3r.$  We have that \begin{align*} 0 & = Y(Y-1)W^{''}(Y) + W^{'}(Y)\left(2rY + \frac{7Y - 6aY - 3}{6}\right) \\
 & \qquad +  W(Y)\left(\frac{(b+c)Y + 3c}{Y-1} + \frac{r(7Y - 6aY - 3)}{6(Y-1)} + \frac{r(r-1)Y}{Y-1}\right)  \\
 & = Y(Y-1)W^{''}(Y) + W^{'}(Y)\left(Y\left(2r + \frac{7 - 6a}{6}\right) - \frac{1}{2}\right) \\ & \qquad +   W(Y)\left(\frac{(Y(6b + 6c + 7r -6ar + 6r(r-1)) + 18c - 3r }{6(Y-1)}\right)  \\
&= Y(Y-1)W^{''}(Y) + W^{'}(Y)\left(Y\left(2r + \frac{7 - 6a}{6}\right) - \frac{1}{2}\right) \\ & \qquad + W(Y)\left(\frac{Y(-18c + 3r) + 18c - 3r}{6(Y-1)}\right) \\
& = Y(Y-1)W^{''}(Y) + W^{'}(Y)\left(Y\left(2r + \frac{7 - 6a}{6}\right) - \frac{1}{2}\right) + W(Y)\left(\frac{r - 6c}{2}\right). \end{align*}

\nin We conclude that 
\begin{equation}
\label{eqn:2F1} 
 0 = Y(Y-1)W^{''}(Y) + W^{'}(Y)\left(Y(2r + \frac{7 - 6a}{6}) - \frac{1}{2}\right) + W(Y)\left(\frac{r - 6c}{2}\right).\
\end{equation}
The differential equation (\ref{eqn:2F1}) is an example of a differential equation in Gauss normal form: 
\begin{equation}
\label{eqn:GNF}
Y(Y-1)W^{''}(Y) + W^{'}(Y)((A + B + 1)Y - C) + AB W(Y) = 0
\end{equation}
If $A - B \not \in \Z$ then a basis for the space of solutions to the differential equation (\ref{eqn:GNF})
in a neighborhood of $\infty$ (see Section $12$ in \cite{ODE}) is $$Y^{-A}\;_{2}F_{1}(A, 1+ A - C, 1 + A - B; Y^{-1}), \textrm{and }Y^{-B}\;_{2}F_{1}(B, 1 + B - C, 1 + B - A;Y^{-1})$$ where $_{2}F_{1}$ denotes the 
Gaussian hypergeometric function defined in equation (\ref{eqn:Gauss2F1}). In our case, 
\begin{equation}
\label{eq:ABC}
A + B + 1 = 2r + \frac{7 - 6a}{6}, \; AB = \frac{r - 6c}{2}, \; C = \frac{1}{2}.
\end{equation} We note that $A$ and $B$ are the roots of the quadratic polynomial $x^{2} - x(2r + \frac{1 - 6a}{6}) + \frac{r -6c}{2}.$
 We recall that $H(Y) = (Y-1)^{r}W(Y)$ and conclude that if $A - B \not \in \Z$ then \newline 
 \nin $(Y-1)^{r}Y^{-A}\;_{2}F_{1}(A, 1+ A - C, 1 + A - B; Y^{-1})$ and  $(Y-1)^{r}Y^{-B}\;_{2}F_{1}(B, 1 + B - C, 1 + B - A;Y^{-1})$ form a basis for the space of solutions to the differential equation (\ref{eqn:ODEmonic}). 
 As $Y = \mathfrak{J}(\tau)$, we have that if $A - B \not \in \Z$ then a basis for the space of solutions in a neighborhood of $\infty$ to the differential equation $$D_0^{2}f + aGD_0f + (bG^2 + cE_4)f = 0$$ is
  $$(\mathfrak{J}(\tau)-1)^{r}\mathfrak{J}(\tau)^{-A} \;_{2}F_{1}(A, 1+ A - C, 1 + A - B; \mathfrak{J}(\tau)^{-1})$$ and
  $$(\mathfrak{J}(\tau)-1)^{r}\mathfrak{J}(\tau)^{-B}\;_{2}F_{1}(B, 1 + B - C, 1 + B - A; \mathfrak{J}(\tau)^{-1}).$$ Finally, we have that if $A - B \not \in \Z$ then a basis for the space of solutions in a neighborhood of $\infty$ to the differential equation $$D_{k_0}^{2}f + aGD_{k_0}f + (bG^2+ cE_4)f = 0$$ is  $\eta^{2k_0}(\tau)(\mathfrak{J}(\tau)-1)^{r}\mathfrak{J}(\tau)^{-A}\;_{2}F_{1}(A, 1+ A - C, 1 + A - B; \mathfrak{J}(\tau)^{-1})$ and 
  \newline \nin $\eta^{2k_0}(\tau)(\mathfrak{J}(\tau)-1)^{r}\mathfrak{J}(\tau)^{-B}\;_{2}F_{1}(B, 1 + B - C, 1 + B - A; \mathfrak{J}(\tau)^{-1}).$ We have thus found a basis of solutions to the differential equation (\ref{eqn:MLDE}) that the component functions of $F$ satisfy.  \\  
  
 \nin  We recall that we have made the stipulation that $\rho(T)$ is diagonalizable in order to avoid logarithmic vector-valued modular forms. This condition on $\rho$ will now play 
an essential role. We claim that the eigenvalues of $\rho(T)$ are distinct. To see this, we use the fact that $\Gamma_{0}(2)$ is generated by $T$ and $V := \begin{bmatrix} 1 & -1 \\ 2 & - 1 \end{bmatrix}.$ We argue by contradiction and suppose that the eigenvalues of $\rho(T)$ are not distinct. Then $\rho(T)$ is a scalar matrix. If $v$ is any nonzero eigenvector for $\rho(V)$ then $\C v$ is a $\Gamma_{0}(2)$-invariant subspace, which contradicts the irreducibility of $\rho.$ Thus the eigenvalues of $\rho(T)$ are distinct. \\

\nin We now use the fact that the eigenvalues of $\rho(T)$ are distinct to show that $A - B \not \in \Z$.  Let $m_1, m_2 \in \C$ such that $|m_1| \leq |m_2|$ and such that the eigenvalues of 
$\rho(T)$ are $e^{2 \pi i m_{1}}$ and $e^{2 \pi i m_{2}}.$ The fact that the eigenvalues of $\rho(T)$ are distinct implies that $m_1 - m_2 \not \in \Z.$
Let $X \in \textrm{GL}_{2}(\C)$ such that 
 $$X \rho(T) X^{-1}  =\left[ \begin{matrix}  e^{2 \pi i m_1} & 0 \\
 0 & e^{2 \pi i m_2} \end{matrix}\right].$$ We recall that $\chi$ denotes the character associated to the modular form $\eta^{2}$ and $\rho_{0} = \rho \otimes \chi^{-k_0}.$
As $\chi(T) = e^{\frac{2 \pi i}{6}},$ $$X \rho_{0}(T) X^{-1}  =\left[ \begin{matrix}  e^{2 \pi i (m_1- \frac{k_0}{6})} & 0 \\
 0 & e^{2 \pi i (m_2-\frac{k_0}{6})} \end{matrix}\right].$$
 We observe that the eigenvalues of $\rho_{0}(T)$ are distinct since the eigenvalues of $\rho(T)$ are distinct. The function $XF_{0}$ is a 
 vector-valued modular form with respect to $X\rho_{0}X^{-1}$ since $F_0$ is a vector-valued modular form with respect to $\rho_{0}$.  Let $h_1$ denote the first and $h_2$ denote the second component function of $XF_{0}$. 
 We have that  $$ \left[ \begin{matrix}  e^{2 \pi i (m_1- \frac{k_0}{6})} & 0 \\
 0 & e^{2 \pi i (m_2-\frac{k_0}{6})} \end{matrix}\right] \left[\begin{matrix} h_1(\tau) \\ h_2(\tau) \end{matrix} \right] = \left[\begin{matrix} h_1(\tau+1) \\ h_2(\tau+1) \end{matrix} \right] .$$
 Let $\widehat{h_{1}}(\tau) := h_1(\tau)e^{-2 \pi i (m_1 - \frac{k_0}{6}) \tau}$ and let $\widehat{h_{2}}(\tau) := h_2(\tau)e^{-2 \pi i (m_2 - \frac{k_0}{6}) \tau}.$ Then $\widehat{h_{1}}(\tau + 1) = \widehat{h_{1}}(\tau)$
 and $\widehat{h_{2}}(\tau + 1) = \widehat{h_{2}}(\tau)$. Therefore there exist sequences $\{a_n\}_{n=-\infty}^{n = \infty},$ and $\{b_n\}_{n=-\infty}^{\infty}$ such that 
$\widehat{h_1}(\tau) = \sum_{n \in \Z} a_n q^{n}$ and $\widehat{h_2}(\tau) = \sum_{n \in \Z} b_n q^{n}$. Thus $h_{1}(\tau) = q^{(m_1 - \frac{k_0}{6})} \sum_{n \in \Z} a_n q^{n}$ and $h_{2}(\tau) = q^{(m_2 - \frac{k_0}{6})} \sum_{n \in \Z} b_n q^{n}.$ 
 As $F$ is holomorphic at $\infty$, $F_0$ is meromorphic at $\infty$ and therefore $a_n  = 0$ if $n < < 0$ and $b_n = 0$ if $n < < 0$. Let $l_1$ and $l_2$ denote the unique complex numbers such that there exist sequences of complex numbers $(c_{n})_{n \geq 0}$ and $(d_{n})_{n \geq 0}$ with $c_0 \neq 0$, $d_0 \neq 0$, and such that 
 $h_{1}(\tau) = q^{l_1} \sum_{n=0}^{\infty} c_{n}q^{n}$ and $h_{2}(\tau) = q^{l_{2}} \sum_{n=0}^{\infty} d_{n}q^{n}.$ We note that $l_1 - (m_1 - \frac{k_{0}}{6}) \in \Z$ and $l_2 - (m_2 - \frac{k_0}{6}) \in \Z$ since $q^{l_1} = q^{m_1  - \frac{k_0}{6}}$ and $q^{l_2} = q^{m_2 - \frac{k_0}{6}}.$
 Hence $l_1 - l_2 \not \in \Z$ since $m_1 - m_2 \not \in \Z.$ \\

\nin The component functions $h_1$ and $h_2$ of $XF_{0}$ are solutions to the differential equation 
 $$D_{2}(D_{0} g) + aG D_{0}g + (bG^{2} + cE_4)g = 0$$ since the component functions of $F_0$ are solutions to this differential equation.
 We note that $h_1$ and $h_2$ cannot be linearly dependent because $l_1 - l_2 \not \in \Z$.  Thus the functions $h_1$ and $h_2$ form a basis for the space of solutions to the above differential equation. If we substitute $h_1 = q^{l_1} \sum_{n=0}^{\infty} c_{n}q^{n}$ into this differential equation then we get that $0 = D_{2}(D_{0} h_1) + aG D_{0} h_1 + (bG^2 + cE_4)h_1 = q^{l_1}(l_1^{2} + (a- \frac{1}{6}) l_1 + b + c + O(q)).$  Hence 
 $l_1^{2} + (a- \frac{1}{6}) l_1 + b + c = 0.$ Similarly, $l_2^{2} + (a - \frac{1}{6})l_2 + b + c = 0.$ Thus $l_1 l_ 2 = b + c$ and $l_1 + l_2 = \frac{1}{6} - a.$ \\
  
  \nin We now relate the numbers $A$ and $B$ to $l_1$ and $l_2$ in order to establish that $A - B \not \in \Z.$  We recall that $A$ and $B$ are the roots of the polynomial $x^{2} - x(2r + \frac{1 - 6a}{6}) + \frac{r - 6c}{2}.$ Let $D$ denote the discriminant of this polynomial. We recall that $l_1 l_ 2 = b + c$ and $l_1 + l_2 = \frac{1}{6} - a$. 
We have that 
 \begin{align*} D & = \left(2r + \frac{1 -6a}{6}\right)^{2}  - 2(r- 6c)  & \\
 & = 4r^2 + \frac{1}{36} + a^{2} - 4ar - \frac{a}{3} - \frac{4r}{3} + 12c  \\
 & = 4\left(r(a + \frac{1}{3}) - (b + 4c)\right)  + \frac{1}{36} + a^{2} - 4ar - \frac{a}{3} - \frac{4r}{3} + 12c  \\
 & = -4(b+c) + (a - \frac{1}{6})^{2} \\
 & = -4 l_1 l_2 + (-l_1 - l_2)^{2} \\
 & = (l_1 - l_{2})^{2}. \end{align*} 
 We use the quadratic formula to see that $A$ and $B$ are the numbers $r + \frac{1}{2}(\frac{1}{6} - a \pm \sqrt{D}) = r + \frac{1}{2}(l_{1} + l_{2} \pm (l_1 - l_2)).$ Thus $\{A, B\} = \{r +  l_1,
 r + l_2\}$. We now fix the values of $A$ and $B$ by choosing to set $A = r +  l_1$ and $B =  r + l_{2}.$  
 Thus $A  - B = l_1 - l_2 \equiv m_1 - m_2 \; (\textrm{mod } \Z)$ since 
 $l_1 \equiv m_1 + \frac{k_{0}}{6} \; (\textrm{mod }  \Z)$ and $l_{2} \equiv m_{2} + \frac{k_0}{6} \; (\textrm{mod } \Z).$ Hence $A - B \not \in \Z$ since $m_1 - m_2 \not \in \Z.$ We now proceed with the proof of \cref{thm: explicitbasis}. We state a more technical version of this theorem below than the statement given in the introduction. \\

\nin \textbf{Theorem 1.5.} Let $\rho$ denote an irreducible complex representation of $\Gamma_{0}(2)$ of dimension two such that $\rho(T)$ is diagonalizable. Let $k_0$ denote the least integer for which $M_{k_0}(\rho) \neq 0$ and let $F$ denote a nonzero element in $M_{k_0}(\rho)$. Let $e^{2 \pi i m_1}$ and $e^{2 \pi i m_2}$ denote the eigenvalues of the matrix $\rho(T)$ with $|m_1| \leq |m_2|$. Let $X \in \textrm{GL}_{2}(\C)$ such that 
 $X \rho(T) X^{-1}  =\left[ \begin{matrix}  e^{2 \pi i m_1} & 0 \\
 0 & e^{2 \pi i m_2} \end{matrix}\right]$.
 Let $a,b,$ and  $c$ denote the unique complex numbers such that $$D_{k_0 +2}(D_{k_0} F)+ aG D_{k_0}F + (bG^2 + cE_4)F = 0.$$
  Let $r$ denote a complex number such that $r(r -1) + (\frac{2 - 3a}{3})r + (b + 4c) = 0.$ Let $A$ and $B$ denote the roots of the quadratic polynomial $x^{2} - x(2r + \frac{1 - 6a}{6}) + \frac{r -6c}{2}.$
  Then there exist unique nonzero complex numbers $\kappa_1$ and $\kappa_2$ such that  
  $$F(\tau) = X^{-1} \left[ \begin{matrix} \kappa_1 \eta^{2k_{0}}(\tau) (\mathfrak{J}(\tau)-1)^{r}\mathfrak{J}(\tau)^{-A}\;_{2}F_{1}(A, \frac{1}{2}+ A , 1 + A - B; \mathfrak{J}(\tau)^{-1})   \\  \\
  \kappa_2  \eta^{2k_{0}}(\tau) (\mathfrak{J}(\tau)-1)^{r}\mathfrak{J}(\tau)^{-B}\;_{2}F_{1}(B, \frac{1}{2} + B, 1 + B - A; \mathfrak{J}(\tau)^{-1}) \end{matrix} \right]. $$ 

\nin \begin{remark} We emphasize the matrix $X$ in \cref{thm: explicitbasis} above is different than the matrix $Q$ that appears in the less technical version of \cref{thm: explicitbasis}, which is stated in the introduction. We also note the presence of the constants 
$\kappa_1$ and $\kappa_2$ in the statement of \cref{thm: explicitbasis} above, which do not appear when \cref{thm: explicitbasis} is stated in the introduction. \end{remark}

\begin{proof}  \nin The functions $h_1$ and $h_2$ both have what we call a \textit{pure $q$-expansion}. We say that a function $R$
 has a \textit{pure $q$-expansion} if $R =  q^{\nu}\sum_{n \in \Z} \alpha_{n} q^{n}$ for some complex number $\nu.$ 
If $\alpha_n = 0$ for $n <0$ and $\alpha_0 \neq 0$ then $\nu$ is uniquely determined and we call $\nu$ the \textit{leading exponent} of $R.$
Every solution of the differential equation $D_{2}(D_{0} g) + aG D_{0}g + (bG^{2} + cE_4)g = 0$ is a linear combination of $h_1 = q^{l_1} \sum_{n=0}^{\infty} c_{n}q^{n}$
and $h_2 =  q^{l_{2}} \sum_{n=0}^{\infty} d_{n}q^{n}.$ The fact that $l_1 - l_2 \not \in \Z$ implies that the only solutions to this differential equation which have a pure $q$-expansion are scalar multiples of $h_1$ or scalar multiples of $h_2.$  We have shown that $A - B \not \in \Z$ and thus that $(\mathfrak{J}(\tau)-1)^{r}\mathfrak{J}(\tau)^{-A}\;_{2}F_{1}(A, 1+ A - C, 1 + A - B; \mathfrak{J}(\tau)^{-1})$ and 
$(\mathfrak{J}(\tau)-1)^{r}\mathfrak{J}(\tau)^{-B}\;_{2}F_{1}(B, 1 + B - C, 1 + B - A; \mathfrak{J}(\tau)^{-1})$ form a basis for the space of solutions to the differential equation  $$D_{2}(D_{0} g) + aG D_{0}g + (bG^{2} + cE_4)g = 0.$$ We will show that $(\mathfrak{J}(\tau)-1)^{r}\mathfrak{J}(\tau)^{-A}\;_{2}F_{1}(A, 1+ A - C, 1 + A - B; \mathfrak{J}(\tau)^{-1})$ has a pure $q$-expansion with leading exponent $A - r = l_1$ and  that $(\mathfrak{J}(\tau)-1)^{r}\mathfrak{J}(\tau)^{-B}\;_{2}F_{1}(B, 1 + B - C, 1 + B - A; \mathfrak{J}(\tau)^{-1})$ has  a pure $q$-expansion with leading exponent $B - r = l_2.$ This will imply that there exist
$\kappa_1, \kappa_2 \in \C$ such that $h_1= \kappa_1 (\mathfrak{J}(\tau)-1)^{r}\mathfrak{J}(\tau)^{-A}\;_{2}F_{1}(A, 1+ A - C, 1 + A - B; \mathfrak{J}(\tau)^{-1})$ 
and  $h_2= \kappa_2(\mathfrak{J}(\tau)-1)^{r}\mathfrak{J}(\tau)^{-B}\;_{2}F_{1}(B, 1 + B - C, 1 + B - A; \mathfrak{J}(\tau)^{-1}).$ \\

\nin  We will employ Newton's binomial theorem, which states that if $\alpha \in \C$ and if $|x| < 1$ then $(1 + x)^{\alpha} = \sum_{n=0}^{\infty} \binom{\alpha}{n}x^{n}.$ 
We note that $|q| <1$ because $\tau \in \mathfrak{H}$. This observation will justify our application of Newton's binomial theorem. 
We have that $\mathfrak{J}(q) = \frac{3G^2}{E_4 - G^2} = \frac{1}{64q}(1 + O(q)).$ We now apply Newton's binomial theorem to get that
 for each integer $n$, $\mathfrak{J}(\tau)^{-n} = (64q)^{n}(1 + O(q)).$
Thus  $$_{2}F_{1}(A, 1+ A - C, 1 + A - B; \mathfrak{J}(\tau)^{-1}) = 1 + \sum_{n \geq 1} \frac{(A)_{n}(1 + A - C)_{n}}{(1 + A - B)_{n} n!} \mathfrak{J}(\tau)^{-n} =  1 + O(q)$$ and  
$$_{2}F_{1}(B, 1 + B - C, 1 + B - A; \mathfrak{J}(\tau)^{-1}) = 1 + \sum_{n \geq 1} \frac{(B)_{n}(1 + B - C)_{n}}{(1 + B - A)_{n} n!} \mathfrak{J}(\tau)^{-n} = 1 + O(q).$$ 

\nin We again apply Newton's binomial theorem to get that $$\mathfrak{J}^{-A}(q) = (64q)^{A}(1 + O(q)), \mathfrak{J}^{-B}(q) = (64q)^{B}(1 + O(q)), (\mathfrak{J} - 1)^{r} = (64q)^{-r}(1 + O(q)).$$ It now follows that 
 $$(\mathfrak{J}(\tau)-1)^{r}\mathfrak{J}(\tau)^{-A}\;_{2}F_{1}(A, 1+ A - C, 1 + A - B; \mathfrak{J}(\tau)^{-1}) = (64q)^{A -r}(1 + O(q)) = (64q)^{l_1}(1 + O(q))$$ has a pure $q$-expansion with leading exponent $l_1 = A -r$ and that 
 $$(\mathfrak{J}(\tau)-1)^{r}\mathfrak{J}(\tau)^{-B}\;_{2}F_{1}(B, 1 + B - C, 1 + B - A; \mathfrak{J}(\tau)^{-1}) = (64q)^{B- r}(1 + O(q)) = (64q)^{l_2}(1 + O(q))$$ has a pure $q$-expansion with leading exponent 
$l_2 = B - r.$ Thus there exist unique nonzero complex numbers $\kappa_1$ and $\kappa_2$ such that  
 $$XF_0 = \left[ \begin{matrix} h_1 \\ h_2 \end{matrix} \right] = \left[ \begin{matrix} \kappa_1  (\mathfrak{J}(\tau)-1)^{r}\mathfrak{J}(\tau)^{-A}\;_{2}F_{1}(A, 1+ A - C, 1 + A - B; \mathfrak{J}(\tau)^{-1})   \\  \\  \kappa_2  (\mathfrak{J}(\tau)-1)^{r}\mathfrak{J}(\tau)^{-B}\;_{2}F_{1}(B, 1 + B - C, 1 + B - A; \mathfrak{J}(\tau)^{-1}) \end{matrix} \right].$$ We substitute $C = \frac{1}{2}$ and we get that 
 \begin{align*} F(\tau) & = \eta^{2k_0}(\tau) F_0(\tau) \\
 & = \eta^{2k_0}(\tau) X^{-1}(XF_{0})(\tau) \\ 
 & = X^{-1} \left[ \begin{matrix} \kappa_1 \eta^{2k_{0}}(\tau) (\mathfrak{J}(\tau)-1)^{r}\mathfrak{J}(\tau)^{-A}\;_{2}F_{1}(A, \frac{1}{2}+ A , 1 + A - B; \mathfrak{J}(\tau)^{-1})   \\  \\
  \kappa_2  \eta^{2k_{0}}(\tau) (\mathfrak{J}(\tau)-1)^{r}\mathfrak{J}(\tau)^{-B}\;_{2}F_{1}(B, \frac{1}{2} + B, 1 + B - A; \mathfrak{J}(\tau)^{-1}) \end{matrix} \right]. \end{align*}
\end{proof}

\section{The arithmetic of vector-valued modular forms} \label{arithmetic}
\subsection{The Fourier expansions of the component functions of $F$}

\nin In the previous section, we proved that there exist unique nonzero complex numbers $\kappa_1$ and $\kappa_2$ such that  
  $$F(\tau) = X^{-1} \left[ \begin{matrix} \kappa_1 \eta^{2k_{0}}(\tau) (\mathfrak{J}(\tau)-1)^{r}\mathfrak{J}(\tau)^{-A}\;_{2}F_{1}(A, \frac{1}{2}+ A , 1 + A - B; \mathfrak{J}(\tau)^{-1})   \\  \\
  \kappa_2  \eta^{2k_{0}}(\tau) (\mathfrak{J}(\tau)-1)^{r}\mathfrak{J}(\tau)^{-B}\;_{2}F_{1}(B, \frac{1}{2} + B, 1 + B - A; \mathfrak{J}(\tau)^{-1}) \end{matrix} \right]. $$
 \nin  We wish to study the arithmetic properties of the Fourier coefficients of the component functions of $F.$ These Fourier coefficients need not be algebraic numbers since $\kappa_1$ and $\kappa_2$ may not be algebraic numbers. Instead, we study the $q$-series expansions of the functions:
  $$\eta^{2k_0} (\mathfrak{J}(\tau)-1)^{r}\mathfrak{J}(\tau)^{-A}\;_{2}F_{1}(A, \frac{1}{2}+ A , 1 + A - B; \mathfrak{J}(\tau)^{-1})$$  $$\eta^{2k_{0}}(\mathfrak{J}(\tau)-1)^{r}\mathfrak{J}(\tau)^{-B}\;_{2}F_{1}(B, \frac{1}{2} + B, 1 + B - A; \mathfrak{J}(\tau)^{-1}).$$ We will show that if $\rho$ has certain properties then the $q$-series coefficients of these two functions are algebraic numbers.  
 
\begin{definition}
We let $\{h(K)\}_{K=1}^{\infty}$ and $\{\widetilde{h}(K)\}_{K=1}^{\infty}$ denote the sequences for which 
$$ \mathfrak{J}^{-A} (\mathfrak{J}-1)^{r} \;_{2}F_{1}(A, \frac{1}{2} + A , 1 + A - B; \mathfrak{J}(\tau)^{-1}) = 64^{A -r} q^{A-r}\left(1 + \sum_{K=1}^{\infty} h(K)q^{K}\right)$$ 
and  $$ \mathfrak{J}^{-B} (\mathfrak{J}-1)^{r} \;_{2}F_{1}(B, \frac{1}{2} + B , 1 + B - A; \mathfrak{J}(\tau)^{-1}) = 64^{B -r} q^{B-r}\left(1 + \sum_{K=1}^{\infty} \widetilde{h}(K)q^{K}\right).$$ 
\end{definition} 

\begin{definition}
Let $F' :=  \left[ \begin{matrix}  \eta^{2k_0} q^{A-r}(1 + \sum_{K=1}^{\infty} h(K)q^{K}) \\ \eta^{2k_0} q^{B-r}(1 + \sum_{K=1}^{\infty} \widetilde{h}(K)q^{K}) \end{matrix} \right].$
\end{definition}

\nin The vector-valued function $F'$ may be obtained from $XF$ by normalizing both of the component functions of $XF$ to have their leading Fourier coefficients equal one. 
In fact, $$F' = \left[ \begin{matrix} 64^{r -A} \kappa_1^{-1} & 0 \\ 0 & 64^{r-B} \kappa_2^{-1} \end{matrix} \right] X F.$$ 

\nin We have that $\eta^{2k_0} = q^{\frac{k_{0}}{12}}(1 + O(q))$ since $\eta = q^{\frac{1}{24}}(1 + O(q)).$  We therefore make the following definition: 
\begin{definition} Let $\{d(K) \}_{K=1}^{\infty}$ and  
$\{\widetilde{d}(K) \}_{K=1}^{\infty}$ denote the sequences of numbers for which 
$$F' = \left[ \begin{matrix}  q^{\frac{k_0}{12} + A -r}(1 + \sum_{K=1}^{\infty} d(K)q^{K})  \\  q^{\frac{k_0}{12} + B -r}(1 + \sum_{K=1}^{\infty} \widetilde{d}(K)q^{K}) \end{matrix} \right].$$
\end{definition} 

\nin Thus $$1 + \sum_{K \geq 1} d(K)q^{K} = \eta^{-2k_0} 64^{A-r} \left(1 + \sum_{K \geq 1} h(K)q^{K}\right),$$ 

$$1 + \sum_{K \geq 1} \widetilde{d}(K) q^{K} = \eta^{-2k_0} 64^{B -r}\left(1 + \sum_{K \geq 1} \widetilde{h}(K) q^{K}\right)$$

\begin{definition} Let $E :=  \left[ \begin{matrix} 64^{r -A} \kappa_1^{-1} & 0 \\ 0 & 64^{r-B} \kappa_2^{-1} \end{matrix} \right] X$ and let $\rho' = E \rho E^{-1}.$
\end{definition}
\nin For each $k \in \Z$, the map $Z \mapsto EZ$ gives an isomorphism from $M_{k}(\rho)$ to $M_{k}(\rho')$ and a $M(\Gamma_{0}(2))$-module isomorphism from $M(\rho)$ to $M(\rho').$ Thus $F' \in M_{k_0}(\rho').$ It is convenient to phrase our results in terms of vector-valued modular forms for $\rho'.$ We will show in this section that if $\rho$ has certain properties then for each integer $k$, there is a basis for $M_{k}(\rho')$ whose component functions have the property that all of their Fourier coefficients are algebraic numbers. \\

\nin To effectively study the Fourier coefficients of $F'$, we will give formulas for $h(K)$ and $\widetilde{h}(K)$ in Theorem \ref{sequence}.  In the second part of this section, we will use the formulas in Theorem \ref{sequence} to study the denominators of the Fourier coefficients of the component functions of $F'.$ In particular, we will show that the sequence of the denominators of the Fourier coefficients of each of the component functions of $F'$ is unbounded provided $\rho$ satisfies a certain hypothesis. In the last part of this section, we show that if $\rho$ satisfies this hypothesis then the sequence of the denominators of the Fourier coefficients of the component functions of every vector-valued modular form for $\rho'$ is unbounded provided the Fourier coefficients are algebraic numbers. \\

\nin To give formulas for $h(K)$ and $\widetilde{h}(K)$, it will also be important to use the  Hauptmodul $\mathfrak{K} := 64 \mathfrak{J}$ because  
$\mathfrak{K} \in  \frac{1}{q} \Z[[q]]^{\times}.$ A proof of this fact is given in Lemma \ref{integral} in the appendix.

\nin We will express $h(K)$ and $\widetilde{h}(K)$ in terms of several sequences, which we will now define. Lemma \ref{integral} implies that for each integer $k \geq 0$, $\mathfrak{K}^{-k} \in  q^{k} \Z[[q]]^{\times}.$ We also show in the appendix that $\mathfrak{K} = q^{-1}(1 + O(q)).$ This fact together with Lemma \ref{integral} imply that for each positive integer $t$,  $(q^{-1} \mathfrak{K}^{-1} - 1)^{t} \in q^{t} \Z[[q]].$
\begin{definition}  For each integer $k \geq 0$, let $\{D(s,k)\}_{s=0}^{\infty}$ denote the sequence of integers such that $$\mathfrak{K}^{-k} =  \sum_{s =0}^{\infty} D(s,k)q^{s} =  q^{k} + \sum_{s=k +1}^{\infty} D(s,k) q^{s}.$$ 
\end{definition} 

\begin{definition} For each integer $t > 0$, let $\{C(t,d)\}_{d=0}^{\infty}$ denote the sequence of integers for which $$(q^{-1} \mathfrak{K}^{-1} - 1)^{t} = \sum_{d=0}^{\infty} C(t,d) q^{d} = \sum_{d = t}^{\infty} C(t,d) q^{d}.$$ 
\end{definition}

 \begin{definition} We define $$g(m,n) := \binom{r}{n} \frac{(-1)^{n} 2^{4m + 6n}(2A)_{2m} }{(1 + A - B)_{m} m!},\; \; \widetilde{g}(m,n) := \binom{r}{n} \frac{(-1)^{n} 2^{4m + 6n}(2B)_{2m} }{(1 + B - A)_{m} m!}.$$ 
\end{definition}
\begin{definition} We define
$$f(k) := \sum_{\substack{n,m \geq 0 \\ n + m = k }} g(m,n), \; \; \widetilde{f}(k) := \sum_{\substack{n,m \geq 0 \\ n + m = k }} \widetilde{g}(m,n).$$ 
\end{definition} 

 \begin{thm} \label{sequence}  We have that
\begin{align*} h(K)  & = \sum_{\substack{d, s \geq 0 \\ d+ s = K }} \left(\sum_{t = 0}^{d} C(t,d) \binom{A-r} {t} \right) \left(\sum_{k=0}^{s} f(k) D(s,k)\right) \\
& = f(K) + \sum_{k=0}^{K - 1} f(k)D(s,k) + \sum_{\substack{d, s \geq 0 \\ d+ s = K  \\ s < K}} \left(\sum_{t = 0}^{d} C(t,d) \binom{A-r} {t} \right) \left(\sum_{k=0}^{s} f(k) D(s,k)\right)
\end{align*} 
\nin \textrm{and} 
\begin{align*}  \widetilde{h}(K) & = \sum_{\substack{d, s \geq 0 \\ d+ s = K }} \left(\sum_{t = 0}^{d} C(t,d) \binom{B-r} {t} \right) \left(\sum_{k=0}^{s} \widetilde{f}(k) D(s,k)\right) \\
& = \widetilde{f}(K) + \sum_{k=0}^{K - 1} \widetilde{f}(k)D(s,k) + \sum_{\substack{d, s \geq 0 \\ d+ s = K  \\ s < K}} \left(\sum_{t = 0}^{d} C(t,d) \binom{B-r} {t} \right) \left(\sum_{k=0}^{s} \widetilde{f}(k) D(s,k)\right). 
\end{align*}
\end{thm} 

\begin{proof}
 \nin The proof of the above formula for $\widetilde{h}(K)$ is completely 
analogous to the proof of the formula for $h(K).$ We therefore just give the proof of the formula for $h(K).$ We have that 
 $$_{2}F_{1}(A, \frac{1}{2} + A, 1 + A - B; \mathfrak{J}^{-1}) = 1 + \sum_{m=1}^{\infty} \frac{(A)_{m} (\frac{1}{2} + A)_{m}}{(1 + A - B)_{m} m!} \mathfrak{J}^{-m} =1 + \sum_{m=1}^{\infty} \frac{2^{6m} (A)_{m} (\frac{1}{2} + A)_{m}}{(1 + A - B)_{m} m!} \mathfrak{K}^{-m} $$
 \nin We note that $(A)_{m}(A + \frac{1}{2})_{m} = \left(2^{-m} \prod_{j=0}^{m-1} (2A + 2j)\right)\left( 2^{-m} \prod_{j=0}^{m-1} (2A + 1 + 2j)\right)= 2^{-2m} (2A)_{2m}.$
 Therefore $$_{2}F_{1}(A, \frac{1}{2} + A, 1 + A - B; \mathfrak{J}^{-1}) = 1 +  \sum_{m=1}^{\infty} \frac{2^{4m} (2A)_{2m}}{(1 + A - B)_{m} m!} \mathfrak{K}^{-m}.$$ 
 Similarly, $$_{2}F_{1}(B, \frac{1}{2} + B, 1 + B - A; \mathfrak{J}(\tau)^{-1}) = 1 + \sum_{m=1}^{\infty} \frac{2^{4m} (2B)_{2m}}{(1 + B - A)_{m}m!} \mathfrak{K}^{-m}.$$
 \nin As $\mathfrak{J} = \frac{1}{64q}(1 + O(q))$, $\mathfrak{J}^{-1} = 64q(1 + O(q)).$ We may therefore apply Newton's binomial theorem and we have that: 
 \begin{align*} \mathfrak{J}^{-A} (\mathfrak{J} - 1)^{r} & = \mathfrak{J}^{-A} \mathfrak{J}^{r} (1 -\mathfrak{J}^{-1})^{r} = \mathfrak{J}^{r - A} (1 - \mathfrak{J}^{-1})^{r} \\
 & = \mathfrak{J}^{r - A} \left(1 + \sum_{n =1}^{\infty}  \binom{r}{n} (-1)^{n} \mathfrak{J}^{-n}\right) \\
 & = 64^{A-r} \mathfrak{K}^{r-A} \left(1 + \sum_{n=1}^{\infty} \binom{r}{n} (-1)^{n} 2^{6n} \mathfrak{K}^{-n} \right). \end{align*}
 Thus \begin{align*} & \mathfrak{J}^{-A} (\mathfrak{J}-1)^{r} \;_{2}F_{1}(A, \frac{1}{2} + A , 1 + A - B; \mathfrak{J}(\tau)^{-1})   \\
 & = 64^{A-r} \mathfrak{K}^{r-A} \left(1 + \sum_{n=1}^{\infty} \binom{r}{n} (-1)^{n} 2^{6n} \mathfrak{K}^{-n} \right) \left(1 +  \sum_{m=1}^{\infty} \frac{2^{4m} (2A)_{2m}}{(1 + A - B)_{m} m!} \mathfrak{K}^{-m} \right) \\
 & = 64^{A-r} \mathfrak{K}^{r-A}  \left(1 + \sum_{k=1}^{\infty} \left(\sum_{\substack{n,m \geq 0 \\ n + m = k}} \binom{r}{n} \frac{(-1)^{n} 2^{4m + 6n}(2A)_{2m}}{(1 + A - B)_{m} m!}\right) \mathfrak{K}^{-k} \right)  \end{align*}
We recall that \begin{align*} g(m,n) := \binom{r}{n} \frac{(-1)^{n} 2^{4m + 6n}(2A)_{2m} }{(1 + A - B)_{m} m!} & = \frac{(-1)^{n}(-r)_{n}}{n!} \cdot \frac{(-1)^{n} 2^{4m + 6n}(2A)_{2m} }{(1 + A - B)_{m} m!} \\
 & =   \frac{2^{4m + 6n}(-r)_{n}(2A)_{2m} }{(1 + A - B)_{m} m! n!}.   \end{align*} 
We also recall that $$f(k) := \sum_{\substack{n,m \geq 0 \\ n + m = k }} g(m,n).$$ We have thus shown that
 $$\mathfrak{J}^{-A} (\mathfrak{J}-1)^{r} \;_{2}F_{1}(A, \frac{1}{2} + A , 1 + A - B; \mathfrak{J}(\tau)^{-1}) = 64^{A-r} \mathfrak{K}^{r -A} \left(1 + \sum_{k=1}^{\infty} f(k) \mathfrak{K}^{-k} \right). $$

   \nin  For each integer $k \geq 0$, we recall that $\{D(s,k)\}_{s=0}^{\infty}$ denotes the sequence of integers such that $$\mathfrak{K}^{-k} =  \sum_{s =0}^{\infty} D(s,k)q^{s} =  q^{k} + \sum_{s=k +1}^{\infty} D(s,k) q^{s}.$$
Therefore \begin{align*} & \mathfrak{J}^{-A} (\mathfrak{J}-1)^{r} \;_{2}F_{1}(A, \frac{1}{2} + A , 1 + A - B; \mathfrak{J}(\tau)^{-1}) \\
 & =  64^{A-r} \mathfrak{K}^{r-A} \left(1 + \sum_{k=1}^{\infty} f(k) \mathfrak{K}^{-k} \right) \\
 & = 64^{A-r} \mathfrak{K}^{r-A}  \left(1 + \sum_{k=1}^{\infty} f(k) \left(q^{k} + \sum_{s = k+1}^{\infty} D(s,k) q^{s}\right) \right)\\
 & =  64^{A-r} \mathfrak{K}^{r-A} \left(1 + \sum_{s=1}^{\infty} q^{s}\left(f(s) + \sum_{k=0}^{s-1} D(s,k) f(k)\right)\right). \end{align*}

 \nin We now compute the $q$-expansion of $\mathfrak{K}^{r-A}$. We may use Newton's binomial theorem because $\mathfrak{K}^{-1} = q(1 + O(q)).$ We let $X(q)$ be the function for which 
 $\mathfrak{K}^{-1} = q(1 + X(q)).$ It follows from Lemma \ref{integral} that $X(q) \in q \Z[[q]].$ We have that $$\mathfrak{K}^{r-A} = (q(1 + X))^{A- r} = q^{A -r} (1 + X)^{A - r} = q^{A - r}\left(1 + \sum_{t =1}^{\infty} \binom{A - r}{t} X^{t}\right).$$ 

 \nin For each positive integer $t$, we recall that $\{C(t,d)\}_{d=0}^{\infty}$ denotes the sequence of integers for which 
 $$(q^{-1} \mathfrak{K}^{-1} - 1)^{t} = X^{t} = \sum_{d=0}^{\infty} C(t,d) q^{d} = \sum_{d = t}^{\infty} C(t,d) q^{d}.$$ Hence 
 \begin{align*}\mathfrak{K}^{r-A} & = q^{A -r} \left(1 + \sum_{t =1}^{\infty} \binom{A - r}{t} X^{t}\right) \\
 & = q^{A - r} \left( 1 +  \sum_{t = 1}^{\infty}  \binom{A - r}{t}\left(\sum_{d=t}^{\infty} C(t,d) q^{d}\right)\right) \\
 & = q^{A - r}\left(1 + \sum_{d=1}^{\infty} q^{d}\left(\sum_{t=0}^{d} C(t,d) \binom{A - r}{t}\right)\right). \end{align*}
 
 \nin We have that 
 \begin{align*}& \mathfrak{J}^{-A} (\mathfrak{J}-1)^{r} \;_{2}F_{1}(A, \frac{1}{2} + A , 1 + A - B; \mathfrak{J}(\tau)^{-1}) \\
& = 64^{A-r} \mathfrak{K}^{r -A} \left(1 + \sum_{s=1}^{\infty} q^{s}\left(f(s) + \sum_{k=0}^{s-1} D(s,k) f(k)\right)\right) \\
& =  64^{A-r} q^{A - r}\left(1 + \sum_{d=1}^{\infty} q^{d}\left(\sum_{t=0}^{d} C(t,d) \binom{A-r}{t}\right)\right) \left(1 + \sum_{s=1}^{\infty} q^{s}\left(f(s) + \sum_{k=0}^{s-1} D(s,k) f(k)\right)\right).  \end{align*} 

 \nin We have thus shown that there exists a sequence $\{h(K)\}_{K=1}^{\infty}$ such that 
$$ \mathfrak{J}^{-A} (\mathfrak{J}-1)^{r} \;_{2}F_{1}(A, \frac{1}{2} + A , 1 + A - B; \mathfrak{J}(\tau)^{-1}) = 64^{A -r} q^{A-r}\left(1 + \sum_{K=1}^{\infty} h(K)q^{K}\right).$$ 

\nin Moreover, we have proven that
\begin{align*} h(K)  & = \sum_{\substack{d, s \geq 0 \\ d+ s = K }} \left(\sum_{t = 0}^{d} C(t,d) \binom{A-r} {t} \right) \left(f(s) + \sum_{k=0}^{s-1} f(k) D(s,k)\right) \\
& = f(K) + \sum_{k=0}^{K - 1} f(k)D(s,k) + \sum_{\substack{d, s \geq 0 \\ d+ s = K  \\ s < K}} \left(\sum_{t = 0}^{d} C(t,d) \binom{A-r} {t} \right) \left(f(s) + \sum_{k=0}^{s-1} f(k) D(s,k)\right). 
\end{align*}

\end{proof} 

\nin We recall that the sequences $\{d(K)\}_{K=1}^{\infty}$ and $\{\widetilde{d}(K)\}_{K=1}^{\infty}$ are defined by the condition: 
$$F' = \left[\begin{matrix} q^{\frac{k_0}{12} + A -r}(1 + \sum_{K=1}^{\infty} d(K)q^{K}) \\   q^{\frac{k_0}{12} + B -r}(1 + \sum_{K=1}^{\infty} \widetilde{d}(K)q^{K}) \end{matrix} \right].$$ \\

\nin  We shall now place some assumptions on $\rho$ to ensure that all of the Fourier coefficients of $F'$ are algebraic numbers. One way to proceed is to study those representations $\rho$ for which $\rho(T)$ has finite order. \textbf{Henceforth, we shall always assume that $\rho(T)$ has finite order.} 
This assumption implies that $\rho(T)$ is diagonalizable.  We recall that $A = r + l_1$, $B = r + l_2$, $l_1 \equiv m_1 + \frac{k_0}{6} \; (\textrm{mod } \Z)$,  and $l_2 \equiv m_2 + \frac{k_0}{6} \; (\textrm{mod } \Z).$ Therefore $A - B = l_1 - l_2 \equiv m_1 - m_2 \; (\textrm{mod } \Z).$  We have previously shown that the irreducibility of $\rho$ implies that $m_1 - m_2 \not \in \Z.$ Thus $A - B \not \in \Z.$ The assumption that $\rho(T)$ has finite order implies that the eigenvalues $e^{2 \pi i m_1}$ and $e^{2 \pi i m_2}$ 
of $\rho(T)$ are roots of unity and that $m_1, m_2 \in \Q.$ Because $m_1, m_2 \in \Q$ and $k_0 \in \Z$, we have that $l_1, l_2 \in \Q.$ Thus $A - B = l_1 - l_2 \in \Q \setminus \Z.$
The fact that $l_1, l_2 \in \Q$ also implies that $\Q(A) = \Q(r) = \Q(B).$ We now give the proof of \cref{thm: algebraicbasis}.

\algebraicbasis*

\begin{proof} 
 We recall that: $$1 + \sum_{K \geq 1} d(K)q^{K} = \eta^{-2k_0} 64^{A-r} \left(1 + \sum_{K \geq 1} h(K)q^{K}\right)$$  
$$1 + \sum_{K \geq 1} \widetilde{d}(K) q^{K} = \eta^{-2k_0} 64^{B -r}\left(1 + \sum_{K \geq 1} \widetilde{h}(K) q^{K}\right)$$
 We have that $\eta^{2k_0} = q^{\frac{k_0}{12}}\prod_{n=1}^{\infty} (1 - q^{n})^{2k_0} \in q^{\frac{k_0}{12}} \Z[[q]]^{\times}.$
Therefore the Fourier coefficients of the component functions of $F'$ are algebraic numbers if for all $K$, $h(K)$ and $\widetilde{h}(K)$ are algebraic numbers. The formulas for $h(K)$ and $\widetilde{h}(K)$ in Theorem \ref{sequence} show that $h(K) \in \Q(A, r) = \Q(r)$ and $\widetilde{h}(K) \in \Q(B, r) = \Q(r).$  Thus if $r \in \overline{\Q}$ then all of the Fourier coefficients of both components of $F'$ are elements of $\Q(r)$ and are therefore algebraic numbers. For each integer $k$, the map  $Z \mapsto EZ$ gives an isomorphism from $M_{k}(\rho)$ to $M_{k}(\rho').$ Thus $M(\rho') =  M(\Gamma_{0}(2))F' \bigoplus M(\Gamma_{0}(2))D_{k_0}F'.$ The fact that the Fourier coefficients of the component functions of $F'$ are elements of $\Q(r)$ together with the fact that $E_2 \in \Z[[q]]$ imply that all of the Fourier coefficients of the component functions of $D_{k_0}F'$ are elements of $\Q(r)$ and are therefore algebraic numbers. Finally, for each integer $k$, there exists a basis of $M_{k}(\Gamma_{0}(2))$ consisting of modular forms with integral Fourier coefficients since $M(\Gamma_{0}(2)) = \C[E_4, G]$ and $E_4, G \in \Z[[q]].$ In fact, a basis for $M_{k}(\rho)$ consisting of vector-valued modular forms whose component functions have Fourier coefficients which are algebraic numbers is $\{G^{a}E_4^{b}F': 2a + 4b = k - k_0, a, b \geq 0, a, b \in \Z\} \bigcup \{G^{a}E_4^{b}D_{k_0}F': 2a + 4b = k - k_0 - 2, a, b \geq 0, a,b \in \Z \}.$

\end{proof}

\subsection{Unbounded Denominators: The Minimal Weight Case}
\nin In this section, we study the arithmetic of the Fourier coefficients of the component functions of $F'$. These Fourier coefficients are algebraic numbers but they need not be rational numbers. We therefore need to define the numerator and the denominator of an algebraic number. Let $\overline{\Z}$ denote the ring of algebraic integers. It is well-known that if $\zeta$ is an algebraic number then there exists a positive integer $N$ such that $N \zeta \in \overline{\Z}.$

\begin{definition} If $\zeta$ is a nonzero algebraic number then \textbf{the denominator of $\zeta$} is the smallest positive integer $Z$ such that 
$Z \zeta \in \overline{\Z}$ and \textbf{the numerator of } $\zeta$ is defined to be the algebraic integer $Z \zeta.$
\end{definition}
\nin We say that an integer $Z$ is \textbf{a denominator of} $\zeta$ if $Z \zeta \in \overline{\Z}.$ The collection of denominators of $\zeta$ form a non-zero ideal of $\Z$ and is therefore generated by a smallest 
positive integer, which is \textbf{the denominator of} $\zeta.$ We observe that there does not exist an integer $j > 1$ which divides both the denominator and numerator of $\zeta$ in the ring $\overline{\Z}$. To see why, we notice that if there exists some integer $j > 1$ which divides the denominator $N$ of $\zeta$ and which also divides $N \zeta$ in the ring $\overline{\Z}$ then
$\frac{N}{j} \zeta \in \overline{\Z}$, which contradicts the minimality of $N.$ 
\begin{definition} Let $p$ denote a prime number. We say that an algebraic number $\zeta$
is \textbf{$p$-integral} if $p$ does not divide the denominator of $\zeta.$ 
\end{definition} 
\nin We shall have occasion to use the following elementary result.

\begin{lemma} \label{ring} Let $p$ denote a prime number. The collection of all algebraic numbers which are $p$-integral form a ring. 
\end{lemma}

 \nin We shall profitably use the following Lemma when studying the denominators of the Fourier coefficients of $F'.$
\begin{lemma} \label{quadratic}  Let $M$ denote a square-free integer. Let $p$ denote an odd prime number for which $M$ is not a quadratic residue mod $p$. Let $X \in\Q(\sqrt{M})$ such that $X \not \in \Q.$ Let $Z$ denote the smallest positive integer 
such that $ZX$ is an algebraic integer and let $Y := ZX$. Let $y$ and $z$ denote the integers for which $Y = \frac{x + y \sqrt{M}}{2}.$  Let $R \in \Q$. If $p \nmid y$ then $p$ does not divide the numerator of any element in the set
$\{(X+R)_{t}: t \geq 1 \}.$ \end{lemma} 
\nin \begin{remark} We note that $y \neq 0$ since $Y \not \in \Q.$ We also note that $y$ and $z$ have the same parity since $Y$ is an algebraic integer. \end{remark}
\begin{proof} Let $\sigma$ denote the non-trivial element in $\textrm{Gal}(\Q(\sqrt{M})/\Q)$ and let $N$ denote the norm map from $\Q(\sqrt{M})$ from $\Q$. Let $\mathcal{O}_{\Q(\sqrt{M})}$ denote the ring of integers of 
$\Q(\sqrt{M}).$ We proceed by contradiction and suppose that there exists some positive integer $t$ such that $p$ divides the numerator of $(X + R)_{t} = (\frac{Y + RZ}{Z} )_{t} = Z^{-t} \prod_{i=0}^{t-1}(Y + RZ + iZ)$ in the ring $\mathcal{O}_{\Q(\sqrt{M})}.$ Then  $p \mid  \prod_{i=0}^{t-1}(Y + RZ + iZ)$ in the ring $\mathcal{O}_{\Q(\sqrt{M})}$ and $p \nmid Z^{t}$ in the ring $\mathcal{O}_{\Q(\sqrt{M})}.$ Thus $p \nmid Z.$ We have that $N(p) = p^{2} \mid \prod_{i=0}^{t-1} N(Y + RZ + iZ)$ in the ring $\Z.$ Thus $p \mid N(Y + RZ + jZ)$ for some integer $j$ with $0 \leq j \leq t -1.$ Therefore 
$$0 \equiv 4 N(Y + RZ + jZ) = 4 N \left(\frac{x}{2} + RZ  + jZ + \frac{y}{2} \sqrt{M}\right) = \left(x + 2(R+j)Z \right)^{2} - My^{2} \; (\textrm{mod } p).$$
As $p \nmid y$, $M$ is a quadratic residue mod $p$. This is a contradiction and our proof is now complete. \end{proof}

\nin We recall that $\Q(A) = \Q(B) = \Q(r)$ and $l_1, l_2 \in \Q$ since $\rho(T)$ has finite order. We note that $a, b + c \in \Q$ since $l_1 l_2 = b + c$ and
$l_1 + l_2 = \frac{1}{6} - a.$ We recall that $r$ satisfies the quadratic equation $r^2 + (\frac{-1 -3a}{3})r + b + 4c = 0.$  We have that $\Q(\frac{-1 -3a}{3}, b + 4c) = \Q(b+ 4c) = \Q(c).$ \newline
\nin Thus $[\Q(c)(r): \Q(c)] \leq 2.$ We are most interested in the case when $c \in \Q.$ If $c \in \Q$ and if $[\Q(r): \Q] = 2$ then we will be able to apply Lemma \ref{quadratic} to analyze the denominators of the Fourier coefficients of $F'.$

\begin{assumption} \label{HYP}
 \textbf{Throughout the rest of this paper, we shall assume that $\rho(T)$ has finite order, $c \in \Q$ and that $[\Q(r): \Q] = 2.$} 
\end{assumption}
\begin{definition} Let $M$ denote the square-free integer for which $\Q(r) = \Q(\sqrt{M})$. 
\end{definition} 
\nin We have previously shown that $A - B \in \Q \setminus \Z.$ We therefore make the following definition. 

\begin{definition} Let $u,v \in \Z$ with $v > 1$, $\textrm{gcd}(u,v) = 1$ such that $A - B = \frac{u}{v}$.
\end{definition} 
\begin{definition}  Let $S$ denote the set of odd prime numbers $p$ for which $M$ is not a quadratic residue mod $p$ and $p \equiv u \; (\textrm{mod } v).$ 
\end{definition} 
\begin{definition} Let $\widetilde{S}$ denote the set of odd prime numbers $p$ for which $M$ is not a quadratic residue mod $p$ and $p \equiv \; -u$ (\textrm{mod }$v$).  
\end{definition}
\nin It follows from the quadratic reciprocity law and Dirichlet's theorem on primes in arithmetic progressions that if $S$ is infinite then $S$ has positive density in the set of prime numbers and if $\widetilde{S}$ is infinite then $\widetilde{S}$ has positive density in the set of prime numbers. \\

\nin We will show that if $S$ is infinite then every sufficiently large element in $S$ divides the denominator of at least one Fourier coefficient of the first component of $F'.$
We will also show that if $\widetilde{S}$ is infinite then every sufficiently large element in $\widetilde{S}$ divides the denominator of at least one Fourier coefficient of the second component of $F'.$ At the end of this section, we will give examples of representations for which we can prove that $S$ and $\widetilde{S}$ are infinite. We begin with the following proposition. 

\begin{proposition} \label{gprop}
Assume that $S$ is an infinite set.  Let $K$ denote an integer such that $p_{K} := u + Kv \in S.$ If $m + n \leq K$ and if $m \neq K$ then for all sufficiently large $K$, $g(m,n)$ is $p_{K}$-integral and $p_{K}$ does not divide the numerator of $g(m,n).$ Consequently, for all sufficiently large $K$, $f(k)$ is $p_{K}$-integral provided $k < K.$
\end{proposition} 

\begin{proof} We recall that $$g(m,n) := \frac{2^{4m + 6n}(-r)_{n}(2A)_{2m} }{(1 + A - B)_{m} m! n!} = \frac{2^{4m + 6n}(-r)_{n}(2A)_{2m} }{(1 + \frac{u}{v})_{m} m! n!} =  
\frac{v^{m} 2^{4m + 6n}(-r)_{n}(2A)_{2m}}{\prod_{j=1}^{m} (u + jv) m! n!}.$$  
We will show that if $K$ is sufficiently large then $p_{K}$ does not divide the numerator of $(-r)_{n}$, $p_{K}$ does not divide the numerator of $(2A)_{2m}$, and $p_{K}$ does not divide
any of the integers $v^{m}$, $2^{4m + 6n}, \prod_{j=1}^{m}(u+jv), m!, n!.$ Lemma \ref{ring} will then imply that $g(m,n)$ is $p_{K}$-integral. \\

\nin The stipulations $m + n \leq K$ and $m \neq K$ imply $m \leq K -1$ and $n \leq K.$ 
In particular, $p_{K} = u + Kv > u + mv \geq u + jv$ for any $j$ with $1 \leq j \leq m$. Thus $p_{K}$ does not divide any of the positive elements in the set $\{u + jv: 1 \leq j \leq m \}.$
We note that $0 \not \in \{u + jv: 1 \leq j \leq m \}$ since $\textrm{gcd}(u,v) = 1$ and $v > 1.$
It is possible that some element(s) in the set  $\{u + jv: 1 \leq j \leq m \}$ are negative since $u$ might be negative. Nevertheless, only finitely many elements in 
the set $\{u + jv: 1 \leq j \leq m\}$ are negative since $v > 0.$ Because $u$ and $v$ are fixed, we may choose a sufficiently large $K$ such that $p_{K}$ does not divide any of the negative elements 
in the set $\{u + jv: 1 \leq j \leq m\}.$ For such a $K$, $p_{K}$ does not divide any element in the set $\{u + jv: 1 \leq j \leq m\}.$ 
Because $p_{K}$ is prime, $p_{K} \nmid \prod_{j=1}^{m} (u + jv).$ In particular, we have shown that $p_{K} \nmid \prod_{j=1}^{K-1}(u + jv)$. \\

\nin If $K$ is sufficiently large then $p_{K} > K$ and thus $p_{K} > m$ and $p_{K} > n.$ Because $p_{K}$ is prime, $p_{K} \nmid m!$ and $p_{K} \nmid n!.$ If $K$ is sufficiently large then $p_{K} \nmid v$ and so $p_{K} \nmid v^{m}$ for any $m.$ We recall that $\Q(A) = \Q(r) = \Q(\sqrt{M}).$
Let $x_1, y_1, x_2, y_2 \in \Z$ such that $2A = \frac{x_1  + y_{1}\sqrt{M}}{2}$ and $-r = \frac{x_2 + y_2 \sqrt{M}}{2}.$ We note that $y_1 \neq 0$ and $y_2 \neq 0$ since $A, r \not \in \Q.$
If $K$ is sufficiently large then $p_{K} \nmid y_1$ and $p_{K} \nmid y_2$  and it then follows from Lemma \ref{quadratic} that $p_{K}$ does not divide the numerator of $(2A)_{2m}$ for any $m$ and $p_{K}$ does not divide the numerator of $(-r)_{n}$ for any $n.$ Moreover, if $K$ is sufficiently large then $p_K$ also does not divide the denominators of $2A$ and $-r.$ Hence $p_{K}$ does not divide the denominators of $(2A)_{2m}$ and $(-r)_{n}$ for any $n$ and $m$. Finally, $p_{K} \nmid 2^{4m + 6n}$ since $p_{K}$ is an odd prime.  We have shown that $p_{K}$ does not divide the numerator of 
$(-r)_{n}(2A)_{2m}$ and that $p_{K}$ does not divide any of the integers $v^{m}$, $2^{4m + 6n}$, $\prod_{j=1}^{m} (u + jv)$, $m!$, and $n!$. We conclude that $g(m,n)$ is $p_{K}$-integral by applying Lemma \ref{ring}. \\ 

\nin  If $k < K$ and if $m + n = k$ then we have shown that $g(m,n)$ is $p_{K}$-integral. 
Hence if $k< K$ then $f(k) := \sum \limits_{\substack{n,m \geq 0 \\ n + m = k }} g(m,n)$ is a sum of $p_{K}$-integral numbers and is therefore $p_{K}$ -integral.
\end{proof} 

\begin{thm} Assume that $\rho$ satisfies Assumption \ref{HYP}. Assume that $S$ is an infinite set. If $K$ is sufficiently large then $f(K)$ is not $p_{K}$-integral and $p_{K}f(K)$ is $p_{K}$-integral.
\end{thm} 
\begin{proof} We have that $f(K) = g(K,0) + \sum \limits_{\substack{m + n \leq K \\ m \neq K}} g(m,n).$ We have shown in Proposition \ref{gprop} that $g(m,n)$ is $p_{K}$-integral if $m + n \leq K$ and if $m \neq K.$
We apply Lemma \ref{ring} to get that $ \sum \limits_{\substack{m + n \leq K \\ m \neq K}} g(m,n)$ is $p_{K}$-integral. It now suffices to show that $g(K,0)$ is not $p_{K}$-integral. 
We note that $g(K,0) = \frac{v^{K} 2^{4K}(2A)_{2K}}{\prod_{j=1}^{K} (u + jv) K!}.$ We have previously shown that if $K$ is sufficiently large then
$v^{K}, 2^{4K}$, $K!$, and $\prod_{j=1}^{K-1} (u + jv)$ are not divisible by $p_{K}$ and that $p_{K}$ does not divide the numerator nor the denominator of $(2A)_{2K}.$ 
Therefore $p_{K} \parallel (u + Kv) \prod_{j=1}^{K-1} (u + jv) = \prod_{j=1}^{K} (u + jv).$ We conclude that $p_{K}$ but not $p_{K}^2$ divides the denominator of $g(K,0)$. Hence $f(K)$ is not $p_{K}$-integral and $p_{K}f(K)$ is $p_{K}$-integral. 
\end{proof}

\begin{thm} \label{UBD} Assume that $\rho$ satisfies Assumption \ref{HYP}.
Assume $S$ is infinite. If $K$ is sufficiently large then the denominator of $h(K)$ is divisible by $p_{K}$. Moreover, $h(J)$ is $p_{K}$-integral if $J < K.$
Hence every prime in $S$ which is sufficiently large divides the denominator of $h(K)$ for some $K$. 
\end{thm} 
\begin{proof}
We recall that:
$$ h(K)   = f(K) + \sum_{k=0}^{K - 1} f(k)D(s,k) + \sum_{\substack{d, s \geq 0 \\ d+ s = K  \\ s < K}} \left(\sum_{t = 0}^{d} C(t,d) \binom{A-r} {t} \right) \left(\sum_{k=0}^{s} f(k) D(s,k)\right).$$
We have shown that $f(K)$ is not $p_{K}$-integral and that $f(k)$ is $p_{K}$-integral if $k < K$. Moreover, the numbers $D(s,k)$ and $C(t,d)$ are integers. 
Thus for all $s \leq K -1$, $\sum_{k=0}^{s} f(k)D(s,k)$ is $p_{K}$-integral. We recall that $A - r = l_1 \in \Q.$
Let $y,z \in \Z$ such that $l_1 = \frac{y}{z}.$ Then $\binom{A - r}{t} = \binom{l_1}{t} = \frac{\prod_{j=0}^{t-1} (y - jz)}{z^{t} t!}.$ We note that $t \leq d \leq K$ and $K < p_K$ if $K$ is sufficiently large.
Therefore $p_{K} \nmid t!$ if $K$ is sufficiently large. We also note that $p_{K} \nmid z$ if $K$ is sufficiently large and thus $p_{K} \nmid z^{t}.$ Thus if $K$ is sufficiently large then $\binom{A-r}{t}$ is $p_{K}$-integral for all $t \leq K.$ Hence for all $d \leq K$, $\sum_{t = 0}^{d} C(t,d) \binom{A-r} {t}$ is $p_{K}$-integral. Thus $h(J)$ is $p_{K}$-integral if $J < K$. 
We have also shown that $$h(K) - f(K) = \sum_{k=0}^{K - 1} f(k)D(s,k) + \sum_{\substack{d, s \geq 0 \\ d+ s = K  \\ s < K}} \bigg(\sum_{t = 0}^{d} C(t,d) \binom{A-r} {t} \bigg) \bigg(\sum_{k=0}^{s} f(k) D(s,k)\bigg)$$ is $p_{K}$-integral. Because $f(K)$ is not $p_K$-integral, $h(K)$ is not $p_K$-integral. 
\end{proof}

\begin{thm}  Assume that $\rho$ satisfies Assumption \ref{HYP}. Let $\widetilde{p_{K}} := -u +Kv$. Assume that $\widetilde{S}$ is infinite. If $K$ is sufficiently large then $\widetilde{p_{K}}$ divides the denominator of $\widetilde{h}(K).$
Moreover, $\widetilde{h}(J)$ is $p_{K}$-integral if $J < K.$
Thus the set of primes that divide the denominator of $\widetilde{h}(K)$ for some $K$ is infinite and has positive density within the set of primes.
\end{thm}
\begin{proof} The proof is completely analogous to the proof of Theorem \ref{UBD}. We note that $\Q(B) = \Q(A) = \Q(r) = \Q(\sqrt{M})$ and $B - A = \frac{-u}{v}.$ We also have that
 $$\widetilde{g}(m,n) := \frac{2^{4m + 6n}(-r)_{n}(2B)_{2m} }{(1 + B - A)_{m} m! n!} = \frac{2^{4m + 6n}(-r)_{n}(2B)_{2m} }{(1 + \frac{-u}{v})_{m} m! n!} =  
\frac{v^{m} 2^{4m + 6n}(-r)_{n}(2A)_{2m}}{\prod_{j=1}^{m} (-u + jv) m! n!}.$$  
\end{proof} 
\nin We recall that $\{d(K) \}_{K=1}^{\infty}$ and  
$\{\widetilde{d}(K) \}_{K=1}^{\infty}$ denote the sequences for which 
$$F' = \left[ \begin{matrix}  q^{\frac{k_0}{12} + A -r}(1 + \sum_{K=1}^{\infty} d(K)q^{K})  \\  q^{\frac{k_0}{12} + B -r}(1 + \sum_{K=1}^{\infty} \widetilde{d}(K)q^{K}) \end{matrix} \right].$$

\begin{thm} \label{fullubd}  Assume that $\rho$ satisfies Hypothesis \ref{HYP}.
If $S$ is infinite then for all sufficiently large $K$, $p_{K}$ divides the denominator of $d(K)$ and $d(i)$ is $p_{K}$-integral if $i < K.$ Thus if $S$ is infinite then the set of primes that divide the denominator of at least one Fourier coefficient of the first component function of $F'$ is infinite and has positive density within the set of primes. If $\widetilde{S}$ is infinite then for all sufficiently large $K$, $\widetilde{p_{K}}$ divides the denominator of $\widetilde{d}(K)$ and $\widetilde{d}(i)$ is $\widetilde{p_{K}}$-integral if $i < K.$ 
Thus if $\widetilde{S}$ is infinite then the set of primes that divide the denominator of at least one Fourier coefficient of the second component function of $F'$ is infinite and has positive density within the set of primes. 
\end{thm}

\begin{proof} We recall that $\eta = q^{\frac{1}{24}} \prod_{n=1}^{\infty} (1 - q^n) = q^{\frac{1}{24}}(1 + O(q)) \in q^{\frac{1}{24}}\Z[[q]]^{\times}.$
Therefore $\eta^{2k_0} = q^{\frac{k_0}{12}}(1 + O(q)) \in q^{\frac{k_0}{12}} \Z[[q]]^{\times}.$  We define the sequence of integers $\{e(K)\}_{K=1}^{\infty}$ by 
setting $\eta^{2k_0} = q^{\frac{k_0}{12}}(1 + \sum_{K=1}^{\infty} e(K)q^{K}).$ We have that \begin{align*} & \eta^{2k_{0}}(\tau) (\mathfrak{J}(\tau)-1)^{r}\mathfrak{J}(\tau)^{-A}\;_{2}F_{1}(A, \frac{1}{2}+ A , 1 + A - B; \mathfrak{J}(\tau)^{-1})  \\
&= q^{\frac{k_0}{12}}(1 + \sum_{i=1}^{\infty} e(i)q^{i}) 64^{A -r} q^{A-r}\left(1 + \sum_{K=1}^{\infty} h(K)q^{K}\right). \end{align*}
Thus $1 + \sum_{K=1}^{\infty} d(K) q^{k} = (1 + \sum_{i=1}^{\infty} e(i)q^{i})(1 + \sum_{K=1}^{\infty} h(K)q^{K}).$ \newline 
\nin Hence $d(K) = h(K) + \sum_{i=0}^{K-1} e(i) h(K - i).$ We note that each $e(K - i)$ is an integer. We have proven that $h(K)$ is not $p_{K}$-integral but $h(i)$ is $p_{K}$-integral if $i < K.$
Therefore $d(i)$ is $p_{K}$-integral if $i < K$ and $d(K)$ is not $p_{K}$-integral.  The proof that if $K$ is sufficiently large then $\widetilde{d}(i)$ is $p_{K}$-integral if $i < K$ and 
$\widetilde{p_{K}}$ divides the denominator of $\widetilde{d}(K)$ is completely analogous. 
\end{proof}

\begin{lemma} \label{reciprocity} If $m_1 - m_2 \in \frac{1}{2} \Z \setminus\Z$ then $S = \widetilde{S}$, $S$ is infinite, and $S$ has density $\frac{1}{2}$ within the set of primes. 
\end{lemma}
\begin{proof}In this case, $v = 2$ and $$S = \widetilde{S} = \{p: p \textrm{ is odd and } M \textrm{ is not a quadratic residue mod } p \}.$$ The quadratic reciprocity law and Dirichlet's theorem on primes in arithmetic progressions imply that since $M$ is not a perfect square, $S$ is an infinite set and has density $\frac{1}{2}$ in the set of primes.
\end{proof}

\begin{lemma} If $\rho$ is induced from a character of $\Gamma(2)$ then $m_1 - m_2 \in \frac{1}{2} \Z\setminus\Z.$
\end{lemma}
\begin{proof} Assume that $\psi$ is a character on $\Gamma(2)$ such that $\rho = \textrm{Ind}_{\Gamma(2)}^{\Gamma_{0}(2)} \psi.$ 
Let $\widetilde{\psi}: \Gamma_{0}(2) \rightarrow \C$ be the function defined by setting $\widetilde{\psi}(g) := \psi(g)$ if $g \in \Gamma(2)$ 
and $\widetilde{\psi}(g) := 0$ if $g \not \in \Gamma(2).$ We note that $\Gamma(2)$ is an index two subgroup of $\Gamma_{0}(2)$ and $\{I, T^{-1}\}$ is a basis of left coset representatives for 
$\Gamma(2)$ in $\Gamma_{0}(2)$. Then $\rho(T)$ is similar to the matrix 
$\left[ {\begin{array}{cc}
     0 &  1 \\
\widetilde{\psi}(T^{2}) & 0 \end{array} } \right].$ Thus the trace of $\rho(T)$ is equal to zero.  As $\rho(T)$ is a diagonalizable matrix with eigenvalues $e^{2\pi i m_1}$ and $e^{2 \pi i m_2}$, we have that 
$0 = \textrm{trace}(\rho(T)) = e^{2 \pi i m_1}+ e^{2 \pi i m_2}.$ Thus $e^{2 \pi i(m_1 - m_2)} = -1.$ Hence $m_1 - m_2 \in \frac{1}{2} \Z\setminus \Z.$
\end{proof} 

\begin{thm} \label{infiniteS} Let $\rho$ denote a two-dimensional irreducible representation of $\Gamma_{0}(2)$ which is induced from a character of $\Gamma(2)$. 
 Assume that $\rho$ satisfies Assumption \ref{HYP}. Then $S$ is infinite and $S$ has density $\frac{1}{2}$ in the set of prime numbers.
Moreover, every sufficiently large prime number in $S$ divides the denominator of at least one Fourier coefficient of the first and of the second component functions of $F'.$
\end{thm}
\begin{proof} The hypothesis that $\rho$ is induced from a character of $\Gamma(2)$ implies $m_1 - m_2 \in \frac{1}{2} \Z \setminus \Z.$ Therefore $S = \widetilde{S}.$  Lemma \ref{reciprocity} implies that $S$ is infinite and has density one half in the set of primes. The conclusion of the theorem now follows from Theorem \ref{fullubd}. 
\end{proof}

\subsection{Unbounded Denominators: The General Case}
\nin In the previous section, we proved that if $\rho$ satisfies Assumption \ref{HYP} and if $S$ and $\widetilde{S}$ are infinite then the denominators of the Fourier coefficients of each of the component functions of $F'$ are unbounded. In this section, we will prove the analogous result for vector-valued modular forms of any weight in Theorem \ref{UBDVVMF}. 
Our method of proof follows very closely  Chris Marks' paper \cite{chrismarks}. In \cite{chrismarks}, Marks proved a similar result for three-dimensional vector-valued modular forms with respect to certain three-dimensional representations of $\Gamma.$ Our proof that the denominators of the Fourier coefficients of each of the component functions of $F'$ are unbounded uses entirely different ideas than those in \cite{chrismarks}. 

\begin{definition} Let $Z$ denote a vector-valued modular form whose component functions have Fourier coefficients which are algebraic numbers. We say that $Z$ has \textbf{bounded denominators} if 
the sequence of the denominators of the Fourier coefficients of each component function of $Z$ are bounded. If $Z$ does not have bounded denominators, we say that $Z$ has \textbf{unbounded denominators}.
\end{definition}
\begin{remark} We proved in \cref{thm: algebraicbasis} that for each integer $k$, there exists a basis for $M_{k}(\rho')$ consisting of vector-valued modular forms whose component functions have Fourier coefficients in $\overline{\Q}.$  \end{remark}

\begin{definition} Let $L$ denote a number field or $\overline{\Q}$. The notation $M_{k}(\rho')_{L}$ denotes those vector-valued modular forms in $M_{k}(\rho')$ whose component functions have all of their Fourier coefficients in $L.$  The notation $M_{k}(\Gamma_{0}(2))_{L}$ denotes those modular forms in $M_{k}(\Gamma_{0}(2))$ whose Fourier coefficients all belong to $L.$ 
We define $M(\rho)_{L} := \bigoplus_{k \in \Z} M_{k}(\rho)_{L}$ and \newline 
$M(\Gamma_{0}(2))_{L} := \bigoplus_{k \in \Z} M_{k}(\Gamma_{0}(2))_{L}.$ 
\end{definition} 

\nin We will need to use the fact that if $f \in  M(\Gamma_{0}(2))_{L}$ then there exists a positive integer $N$ such that $Nf \in \overline{\Z}[[q]].$ To do so, we shall need the following lemma. 
\begin{lemma} \label{lang}  Let $f = \sum_{n=0}^{\infty} a_n q^{n} \in M_{k}(\Gamma_{0}(2)).$ Let $r(k) = \textrm{dim } M_{k}(\Gamma_{0}(2)).$
Let $R$ denote the $\Z[\frac{1}{2}, \frac{1}{3}]$-module generated by $a_0,...,a_{r(k) -1}.$ Then all $a_n \in R$ and 
$f = \sum_{2a + 4b = k} c_{a,b}G^{a}E_4^{b}$ where each $c_{a,b} \in R.$ 
\end{lemma} 

 \nin The proof of this lemma is completely analogous to the proof of  Theorem $4.2$ in Lang's book on modular forms \cite{lang}. 

\begin{lemma} \label{boundedmodular} Let $L$ denote a number field or $\overline{\Q}.$ Let $M_{k}(\Gamma_{0}(2))_{L} :=  M_{k}(\Gamma_{0}(2)) \cap L[[q]].$
Then $\{G^a E_4^{b}: a,b \geq 0, 2a + 4b = k \}$ is a basis for the $L$-vector space $M_{k}(\Gamma_{0}(2))_{L}.$ If $f \in M_{k}(\Gamma_{0}(2))_{L}$ then there exists a positive integer $N$ 
such that $Nf \in \overline{\Z}[[q]].$
\end{lemma} 
\begin{proof} The fact that $\{G^a E_4^{b}: a,b \geq 0, 2a + 4b = k \}$ is a basis for the $L$-vector space $M_{k}(\Gamma_{0}(2))_{L}$ is immediate from Lemma \ref{lang}. Let  $f \in M_{k}(\Gamma_{0}(2))_{L}.$  \newline
\nin Then $f = \sum_{2a + 4b = k} c_{a,b} G^a E_4^{b}$  where each $c_{a,b} \in \C.$ Lemma \ref{lang} and the fact that $f \in M_{k}(\Gamma_{0}(2))_{L}$ imply that each $c_{a,b} \in L.$ Therefore there exists a positive integer $N_{a,b}$ such that 
$N_{a,b} c_{a,b} \in \overline{\Z}.$ Let $N = \prod_{2a + 4b = k} N_{a,b}.$ Then each $N c_{a,b} \in \overline{\Z}$ and $Nf = \sum_{2a + 4b = k} Nc_{a,b} G^a E_4^{b} \in \overline{\Z}[[q]].$
\end{proof} 

\nin We shall need to compute the $q$-series expansion of $D_{k_0}F'$ in this section and we do so in the following Lemma. 
\begin{lemma} \label{t1} Let $t_{1}(K)  :=  d(K)(K + A -r) + \sum_{n=1}^{K-1} 2k_0 \sigma(n)d(K -n)$ and let $t_{2}(K) :=  \widetilde{d}(K)(K + B -r) + \sum_{n=1}^{K-1} 2 k_{0} \sigma(n)\widetilde{d}(K-n).$
Then $$D_{k_0}F' =  \left[ \begin{matrix}  q^{\frac{k_0}{12} + A -r}(A-r  + \sum_{K=1}^{\infty} t_1(K)q^{K})  \\  q^{\frac{k_0}{12} + B -r}(B -r + \sum_{K=1}^{\infty} t_{2}(K) q^{K}) \end{matrix} \right].$$ Moreover, for all integers $K$, $t_{1}(K), t_{2}(K) \in L.$ In particular, $D_{k_0}F' \in M_{k_0+2}(\rho')_{L}.$
\end{lemma} 
\begin{proof} 
\nin We recall that $E_2 = 1 - 24 \sum_{n=1}^{\infty} \sigma(n)q^{n}.$ Let $F_1'$ and $F_2'$ denote the first and second component functions of $F'.$
We have that \begin{align*} D_{k_0}(F_1') & = D_{k_0}\left(q^{\frac{k_0}{12} + A - r}(1 + \sum_{K=1}^{\infty} d(K)q^{K})\right) \\
& =  \theta\left(q^{\frac{k_0}{12} + A - r}(1 + \sum_{K=1}^{\infty} d(K)q^{K})\right) \\ 
& \quad \quad - \frac{k_0}{12}E_2 \left(q^{\frac{k_0}{12} + A - r}(1 + \sum_{K=1}^{\infty} d(K)q^{K})\right) \\
& = q^{\frac{k_0}{12} + A-r} \theta \left(1 + \sum_{K=1}^{\infty} d(K)q^{K} \right) + \left(1 + \sum_{K=1}^{\infty} d(K)q^{K}\right) \theta (q^{\frac{k_0}{12} + A - r}) \\
& \quad \quad -  \frac{k_0}{12}(1 - 24\sum_{n=1}^{\infty} \sigma(n)q^{n}) \left(q^{\frac{k_0}{12} + A - r}(1 + \sum_{K=1}^{\infty} d(K)q^{K})\right) \\
& =  q^{\frac{k_0}{12} + A-r} (\sum_{K=1}^{\infty} K d(K) q^{K}) + (\frac{k_0}{12} + A -r) q^{\frac{k_0}{12} + A -r} \left(1 + \sum_{K=1}^{\infty} d(K)q^{K}\right) \\
& \quad \quad - \frac{k_0}{12} q^{\frac{k_0}{12} + A - r}(1 + \sum_{K=1}^{\infty} (d(K) - \sum_{n=1}^{K-1} 24 \sigma(n)d(K-n)) q^{K}) \\
& = q^{\frac{k_0}{12} + A - r}(A - r + \sum_{K =1}^{\infty} t_{1}(K) q^{K}) \end{align*} 
\nin In a similar manner, we have that $D_{k_0}F_2' = q^{\frac{k_0}{12} + B -r}(B -r + \sum_{K=1}^{\infty} t_{2}(K)q^{K}).$ \\
\nin The assumption that $\rho(T)$ has finite order implies that $A - r \in \Q$ and $B -r \in \Q.$ We have previously shown that for all integers $K \geq 1$, 
$d(K) \in L$ and  $\widetilde{d}(K) \in L$. It now follows from the formulas for $t_1(K)$ and $t_{2}(K)$ that for each integer $K \geq 1$, $t_{1}(K) \in L$ and $t_{2}(K) \in L.$ Hence all of the Fourier coefficients of both of the component functions of $D_{k_0}F'$ belong to $L.$
\end{proof} 

\begin{lemma} \label{freemoduleL}  Assume that $\rho$ satisfies Assumption \ref{HYP}. Let $M$ denote the square-free integer for which $\Q(\sqrt{M}) = \Q(r)$ and let $L = \Q(\sqrt{M}).$ \\
\nin Then $M(\rho')_{L} = M(\Gamma_{0}(2))_{L} F' \bigoplus M(\Gamma_{0}(2))_{L} D_{k_0} F'.$ In particular, $M(\rho')_{L}$ is a free $M(\Gamma_{0}(2))_{L}$-module of rank two. 
\end{lemma}
 \begin{proof} This proof follows the proof of Lemma $4.1$ in Marks' paper \cite{chrismarks}. 
 We have shown in \cref{thm: algebraicbasis} that $F' \in M_{k_0}(\rho')_{L}$ and $D_{k_0}F' \in M_{k_0 +2}(\rho')_{L}.$ \newline 
 \nin Hence $M(\Gamma_{0}(2))_{L} F' \bigoplus M(\Gamma_{0}(2))_{L} D_{k_0} F' \subset M(\rho')_{L}.$ To prove the theorem, we need to show that
 the reverse inclusion holds.  We recall that $M(\rho')_{L} := \bigoplus_{k \in \Z} M_{k}(\rho')_{L}.$ Thus it suffices to prove that if $k \in \Z$ then 
 $M_{k}(\rho')_{L} \subset M(\Gamma_{0}(2))_{L} F' \bigoplus M(\Gamma_{0}(2))_{L} D_{k_0} F'.$ Let $Z \in M_{k}(\rho')_{L}.$ Let $Z_1$ and $Z_2$ denote the first and second component functions of $Z$ and let $F_1'$ and $F_2'$ denote the first and second component functions of $F'.$ We know that \newline
\nin  $M(\rho') = M(\Gamma_{0}(2)) F' \bigoplus M(\Gamma_{0}(2)) D_{k_0} F'.$ Therefore $Z = m_1 F' + m_2 D_{k_0} F'$ where 
\newline \nin $m_1 \in  M_{k - k_0}(\Gamma_{0}(2)), m_2 \in M_{k - k_0 - 2}(\Gamma_{0}(2)).$ It suffices to show that $m_1 \in  M_{k - k_0}(\Gamma_{0}(2))_{L}$ and $m_2 \in M_{k - k_0 - 2}(\Gamma_{0}(2))_{L}.$  We have previously shown that  $$F' = \left[\begin{matrix} F_1' \\ F_2' \end{matrix} \right] =  \left[ \begin{matrix}  q^{\frac{k_0}{12} + A -r}(1 + \sum_{K=1}^{\infty} d(K)q^{K})  \\  q^{\frac{k_0}{12} + B -r}(1 + \sum_{K=1}^{\infty} \widetilde{d}(K)q^{K}) \end{matrix} \right]$$ and 
 $$D_{k_0}(F') = \left[\begin{matrix} D_{k_0} F_1' \\ D_{k_0}F_2' \end{matrix} \right] =  \left[ \begin{matrix}  q^{\frac{k_0}{12} + A -r}(A-r  + \sum_{K=1}^{\infty} t_1(K)q^{K})  \\  q^{\frac{k_0}{12} + B -r}(B -r + \sum_{K=1}^{\infty} t_{2}(K) q^{K}) \end{matrix} \right].$$
 
\nin  We write $m_1 = \sum_{n=0}^{\infty} m_1(n)q^{n}$ and $m_2 = \sum_{n=0}^{\infty} m_2(n)q^{n}.$ The equation $Z = m_1 F' + m_2 D_{k_0}F'$ implies that there exist sequences $\{Z_1(n)\}_{n=0}^{\infty}$ and $\{Z_2(n)\}_{n=0}^{\infty}$ such that 
$$Z = \left[\begin{matrix} Z_1 \\ Z_2 \end{matrix} \right] = \left[\begin{matrix}  q^{\frac{k_0}{12} + A -r }\sum_{n=0}^{\infty} Z_1(n)q^{n}\\ 
q^{\frac{k_0}{12} + B -r} \sum_{n=0}^{\infty} Z_{2}(n) q^{n} \end{matrix} \right].$$
In fact, \begin{align*} Z_1 &= q^{\frac{k_0}{12} + A -r} \sum_{n=0}^{\infty} Z_1(n)q^{n} \\
&= m_1 F_1' + m_2 D_{k_0} F_1 ' \\
&= q^{\frac{k_0}{12} + A - r} \sum_{n=0}^{\infty} m_1(n)q^{n}(1 + \sum_{K=1}^{\infty} d(K)q^{K}) 
+ q^{\frac{k_0}{12} + A - r} \sum_{n=0}^{\infty} m_{2}(n)q^{n}(A - r + \sum_{K=1}^{\infty} t_{1}(K)q^{K}). \end{align*} 
Thus $Z_1(N) = m_1(N) + \sum_{n=0}^{N-1}m_{1}(n) d(N- n) + (A-r)m_2(N)+ \sum_{n=0}^{N-1} m_{2}(n)t_{1}(N -n).$
Similarly, $Z_2(N) = m_1(N) + \sum_{n=0}^{N-1}m_1(n) \widetilde{d}(N-n) + (B-r)m_2(N) + \sum_{n=0}^{N-1} m_{2}(n) t_{2}(N- n).$
Hence $\left[ \begin{matrix} Z_1(0) \\ Z_2(0) \end{matrix} \right] = \left[\begin{matrix} 1 & A -r \\ 1 & B -r \end{matrix} \right] \left[\begin{matrix} m_1(0) \\ m_2(0) \end{matrix} \right]$ and for all $N \geq 1$, we have that
$$\left[ \begin{matrix} Z_1(N) \\ Z_2(N) \end{matrix} \right] = \left[\begin{matrix} 1 & A -r \\ 1 & B -r \end{matrix} \right] \left[\begin{matrix} m_1(N) \\ m_2(N) \end{matrix} \right] + \sum_{n=0}^{N-1} \left[\begin{matrix} m_1(n)d(N-n) + m_2(n) t_1(N -n) \\ m_1(n) \widetilde{d}(N -n) + m_2(n) t_2(N - n) \end{matrix} \right].$$
To show that $m_1, m_2 \in M(\Gamma_{0}(2))_{L}$, we must show that for all nonnegative integers $N$, \newline 
\nin $m_1(N) \in L$ and $m_2(N) \in L.$ We proceed by induction on $N$. Our inductive hypothesis is that for all nonnegative integers $n < N$, $m_1(n), m_2(n) \in L.$  We recall that because $\rho$ is irreducible, $B - A \not \in \Z$ and thus $B - A \neq 0.$
Hence the matrix $\left[\begin{matrix} 1 & A -r \\ 1 & B -r \end{matrix} \right]$ is invertible. The assumption that $\rho(T)$ has finite order implies that $A - r, B -r \in \Q.$ Thus  $$\left[\begin{matrix} 1 & A -r \\ 1 & B -r \end{matrix} \right]^{-1} = \frac{1}{B-A} \left[\begin{matrix} B-r & r - A   \\  -1 & 1 \end{matrix} \right] \in \textrm{GL}_{2}(\Q).$$ 
\nin The assumption that $Z \in M(\rho)_{L}$ implies that for all integers $n \geq 0$, $Z_1(n), Z_2(n) \in L.$ We have previously shown that $F' \in M(\rho)_{L}$ and that for all integers $K \geq 1$, $d(K), \widetilde{d}(K) \in L$
We also proved in Lemma \ref{t1} that for all integers $K$, $t_{1}(K), t_{2}(K) \in L.$ We now treat the base case where $N = 0.$
We have that  $$\left[ \begin{matrix} m_1(0) \\ m_2(0) \end{matrix} \right] = \frac{1}{B -A} \left[\begin{matrix} B-r & r - A \\ -1 & 1 \end{matrix} \right] \left[\begin{matrix} Z_1(0) \\ Z_2(0) \end{matrix} \right].$$ 
\nin Hence $m_1(0), m_2(0) \in L$ since $Z_1(0), Z_2(0) \in L$, $B -r, r - A, B - A \in \Q.$ Let $N$ denote a positive integer. Assume that for all nonnegative integers $n$ with $n < N$, $m_1(n), m_2(n) \in L.$ \\
\nin Then $$\left[\begin{matrix} 1 & A -r \\ 1 & B -r \end{matrix} \right] \left[\begin{matrix} m_1(N) \\ m_2(N) \end{matrix} \right] =   \left[ \begin{matrix} Z_1(N) \\ Z_2(N) \end{matrix} \right] - \sum_{n=0}^{N-1} \left[\begin{matrix} m_1(n)d(N-n) + m_2(n) t_1(N -n) \\ m_1(n) \widetilde{d}(N -n) + m_2(n) t_2(N - n) \end{matrix} \right]  \in L^{2}.$$ 
Thus $m_1(N), m_2(N) \in L$ since $\left[\begin{matrix} 1 & A -r \\ 1 & B -r \end{matrix} \right] \in \textrm{GL}_{2}(\Q).$ This completes the inductive step and our proof is now complete. 
\end{proof}

 \begin{thm}  \label{UBDVVMF} Let $\rho$ denote a two-dimensional irreducible representation of $\Gamma_{0}(2)$ for which $\rho(T)$ has finite order, $c \in \Q$
 and $[\Q(r): \Q] = 2.$ Let $k \in \Z$ and let $Z$ denote a nonzero element in $M_{k}(\rho')_{L}.$ Let $Z_1$ and $Z_2$ denote the first and second component functions of $Z$.  If $S$ is infinite then any sufficiently large prime in $S$ divides the denominator of at least one Fourier coefficient of $Z_1.$ In particular, if $S$ is infinite then the sequence of the denominators of the Fourier coefficients of $Z_1$ is unbounded. If $\widetilde{S}$ is infinite then any sufficiently large prime in $\widetilde{S}$ divides the denominator of some Fourier coefficient of $Z_2.$ In particular, if $\widetilde{S}$ is infinite then the sequence of the denominators of the Fourier coefficients of $Z_2$ is unbounded.   \end{thm}

\begin{proof} This proof follows closely the proof of Proposition $4.3$ in Marks' paper \cite{chrismarks}.
Let $Z$ denote a nonzero element in $M_{k}(\rho')_{L}.$ Then $D_{k}Z \in M_{k+2}(\rho')_{L}.$ We have proven in Lemma \ref{freemoduleL}
 that $$M(\rho')_{L} = M(\Gamma_{0}(2))_{L}  F' \bigoplus M(\Gamma_{0}(2))_{L} D_{k_0}F'.$$
Therefore there exist $m_1, m_4 \in M_{k - k_0}(\Gamma_{0}(2))_{L}$, $m_2 \in M_{k-k_0 -2} (\Gamma_{0}(2))_{L}$, and \newline  $m_3 \in M_{k +2 - k_0}(\Gamma_{0}(2))_{L}$ such that 
$$ \left[ \begin{matrix} Z \\ D_{k}Z \end{matrix} \right] = \left[ \begin{matrix} m_1 & m_2 \\ m_3 & m_4 \end{matrix} \right] \left[ \begin{matrix} F' \\ D_{k_0}F' \end{matrix} \right].$$
We note that $m_1m_4 - m_2 m_3 \in M_{2k - 2k_0}(\Gamma_{0}(2))_{L}.$ \\
\nin We have that:
\begin{align*} \left[ \begin{matrix} m_4Z - m_2 D_{k}Z \\ -m_3 Z + m_1 D_{k}Z \end{matrix} \right] 
& = \left[ \begin{matrix} m_4 & -m_2 \\  -m_3 & m_1 \end{matrix} \right] \left[ \begin{matrix} Z \\ D_{k}Z \end{matrix} \right]  \\
& =  \left[ \begin{matrix} m_4 & -m_2 \\  -m_3 & m_1 \end{matrix} \right] \ \left[ \begin{matrix} m_1 & m_2 \\ m_3 & m_4 \end{matrix} \right] \left[ \begin{matrix} F' \\ D_{k_0}F' \end{matrix} \right]  \\
& = \left[ \begin{matrix}(m_1m_4 - m_2 m_3)F' \\ (m_1m_4 - m_2 m_3) D_{k_0}F'  \end{matrix} \right]. \end{align*}

\nin Thus  $m_4Z - m_2 D_{k}Z = (m_1m_4 - m_2m_3)F'.$ We recall that $Z_1$ and $Z_2$ denote the first and second component functions of $Z$ and $F_1'$ and $F_2'$ denote the first and second component functions of $F'.$ Thus $m_4Z_1 - m_2 D_{k}Z_1 = (m_1m_4 - m_2m_3)F_1'$ and $m_4Z_2 - m_2 D_{k}Z_2 = (m_1m_4 - m_2 m_3)F_2'.$ Let $N$ denote a positive integer for which $Nm_1, Nm_2, Nm_3, Nm_4 \in \overline{\Z}[[q]].$ Then $N^2(m_1m_4 - m_2 m_3) \in \overline{\Z}[[q]]$ and $N^{2}(m_4 Z_1 - m_2 D_{k} Z_1) = N^{2}(m_1 m_4 - m_2 m_3)F_1'.$ \\

\nin Let $p$ denote a prime number in $S$ for which $p > N$, $p \nmid 12$ and $p$ divides the denominator of some Fourier coefficient of $F_1'.$ We remark that any sufficiently large prime in $S$ has this property. Suppose by way of contradiction that $p$ does not divide the denominator of any Fourier coefficient of $Z_1.$
Then $p$ does not divide the denominator of any Fourier coefficient of  $D_{k}Z_1 = q \frac{d}{dq}(Z_1) - \frac{k}{12}E_2 Z_1$ since $E_2 \in \Z[[q]].$ 
Therefore $p$ does not divide the denominator of any Fourier coefficient of  $N^{2} m_4 Z_1 - N^{2} m_2 D_{k} Z_1 = N^2(m_1 m_4 - m_2 m_3)F_1'.$
We recall that $F_1' = q^{\frac{k_0}{12} + A -r}(1 + \sum_{i=1}^{\infty} d(i) q^{i})$ where each $d(i) \in \overline{\Z}.$ We may write $N^2(m_1 m_4  - m_2 m_3) = \kappa q^{t} (1 + \sum_{i =1}^{\infty} a(i)q^{i})$ where $\kappa, a(i) \in \overline{\Z}$ and $t \in \N \cup \{0\}.$  Let $\{x(i)\}_{i=1}^{\infty}$ denote the sequence for which 
 $$N^{2}(m_1 m_4 - m_2 m_3)F_1' = \kappa q^{\frac{k_0}{12} + A - r + t} (1 + \sum_{i = 1}^{\infty} a(i)q^{i})(1 +  \sum_{i=1}^{\infty} d(i)q^{i})
 =q^{\frac{k_0}{12} + A -r + t}  \sum_{i=0}^{\infty} x(i) q^{i}.$$ 
 
\nin Let $K$ denote the positive integer for which $d(K)$ is not $p$-integral and $d(i)$ is $p$-integral for $i < K.$
Then $x(K) = \kappa (d(K) + \sum_{j=1}^{K-1} a(j) d(K - j)).$ We must have that  $\sum_{j=1}^{K-1} a(j) d(K - j))$ is $p$-integral since each $a(j) \in \overline{\Z}$ and $d(i)$ is $p$-integral for $i< K.$ Thus $d(K) + \sum_{j=1}^{K-1} a(j) d(K - j)$ is not $p$-integral since $d(K)$ is not $p$-integral. 
Finally, we stipulate that $p$ has the additional property that $p \nmid \kappa$ in the ring $\overline{\Z}.$ This is the case when $p$ is sufficiently large. Then $x(K)$ is not $p$-integral. 
We have thus shown that $p$ divides the denominator of some Fourier coefficient of $N^2(m_1 m_4 - m_2 m_3)F_1'.$
We have previously shown that if $p$ does not divide the denominator of any Fourier coefficient of $Z_1$ then $p$ does not divide the denominator of any Fourier coefficient of 
$ N^{2} m_4 Z_1 - N^{2} m_2 D_{k} Z_1 = N^2(m_1 m_4 - m_2 m_3)F_1'.$ This is a contradiction. Hence $p$ must divide the denominator of at least one Fourier coefficient of $Z_1.$
In particular, if $S$ is infinite then the sequence of denominators of the Fourier coefficients of $Z_1$ is unbounded. The proof that if $\widetilde{S}$ is infinite then every sufficiently large prime in $\widetilde{S}$ 
divides the denominator of some Fourier coefficient of $Z_2$ is completely analogous. 
\end{proof} 

\begin{restatable}{thm}{oldinducedmain}  \label{thm: oldinducedmain} Let $\rho$ denote a two-dimensional irreducible representation of $\Gamma_{0}(2)$ which is induced from a character of $\Gamma(2)$.  Assume that $\rho(T)$ has finite order, $c \in \Q$ and $[\Q(r): \Q] = 2.$ 
Let $k \in \Z$ and $Z \in M_{k}(\rho')$ whose component functions $Z_1$ and $Z_2$ have the property that all of their Fourier coefficients are algebraic numbers. Then the sequence of the denominators of the Fourier coefficients of $Z_1$ and the sequence of the denominators of the Fourier coefficients of $Z_2$ are unbounded. \\

\nin Let $M$ denote the square-free integer such that $\Q(\sqrt{M}) = \Q(r).$ If $p$ is any sufficiently large prime such that $M$ is not a quadratic residue mod $p$ then $p$ divides the denominator of at least one Fourier coefficient of $Z_1$ and $p$ divides the denominator of at least one Fourier coefficient of $Z_2.$
\end{restatable}

\begin{proof} We have shown in the proof of Theorem \ref{infiniteS} that if $\rho$ is induced by a character of $\Gamma(2)$ then $S = \widetilde{S} = \{p: p \textrm{ is odd and } M \textrm{ is not a quadratic residue mod } p \}.$ Therefore $S$ and $\widetilde{S}$ are infinite sets. The conclusion of \cref{thm: oldinducedmain} is now a consequence of Theorem \ref{UBDVVMF}. 
\end{proof} 
\vspace{0.1 in}
\nin Our proof of \cref{thm: inducedmain} is obtained by combining \cref{thm: oldinducedmain} and 
\cref{thm: examples}. We remind the reader of the notation:   $T =  \left[ {\begin{array}{cc}
   1 & 1 \\
   0 & 1 \\
\end{array} } \right],$  
\nin $U =  \left[ {\begin{array}{cc}
   1 & 0 \\
   1 & 1 \\
\end{array} } \right],$ and
$S = \left[ \begin{array}{cc}  0 & 1 \\ -1 & 0 \end{array} \right ].$ 
We recall that  $\Gamma(2) = \langle T^2$ , $U^{2}, -I \rangle$ and 
$\Gamma_{0}(2) = \langle  T, U^2, -I \rangle.$ 

\begin{restatable}{thm}{examples} \label{thm: examples} Let $\psi$ denote a character of $\Gamma(2)$ such that $\psi(-I) = 1$ and let  $\rho = \textrm{Ind}_{\Gamma(2)}^{\Gamma_{0}(2)} \psi$. Let $\xi_{1}, \xi_{2} \in \C$ such that $e^{2 \pi i \xi_{1}} = \psi(T^2)$ and $e^{2 \pi i \xi_{2}} = \psi(U^2).$ If $\xi_{1} \in \Q$ and if $[\Q(\xi_{2}): \Q] = 2$ then $\rho$ is irreducible, $\rho(T)$ has finite order, the image of $\rho$ is infinite, and $a,b,c \in \Q.$ Moreover, $\Q(r) = \Q(\xi_{2})$ and thus $r$ is an algebraic number of degree two. 
\end{restatable}

\nin Our proof of \cref{thm: examples} makes use of the fact that  $G|_{2} S = - \frac{1}{2} + O(q_2),$ where $q_2 := e^{\pi i \tau}.$   This computation appears in the last subsection of the appendix. \\

\begin{proof}  Let $\alpha = \psi(T^2)$ and let $\beta = \psi(U^2).$ Let $\tilde{\psi}$ denote the function on $\Gamma_{0}(2)$ for which $\tilde{\psi}|_{\Gamma(2)} = \psi$ 
and $\tilde{\psi}(g) = 0$ if $g \not \in \Gamma(2).$ We now compute $\rho$ with respect to the left coset representatives $I, T^{-1}$ of  $\Gamma(2)$ in $\Gamma_{0}(2).$ We have that 
\begin{align*}
\rho(T) & = \begin{bmatrix}\tilde{\psi}(I^{-1}TI) & \tilde{\psi} (I^{-1} T T^{-1}) \\
						\tilde{\psi} ((T^{-1})^{-1}TI) & \tilde{\psi}((T^{-1})^{-1}TT^{-1})  \end{bmatrix} \\
&=  \begin{bmatrix} 0 & 1 \\
				    \alpha & 0  \end{bmatrix}. \end{align*}

 We will use the fact that $TU^2T^{-1} = -T^{2}U^{-2}$ to compute $\rho(U^2).$
We have that  
\begin{align*}
\rho(U^2) & = \begin{bmatrix} \tilde{\psi}(I^{-1}U^2 I) & \tilde{\psi} (I^{-1} U^2 T^{-1}) \\
						\tilde{\psi} ((T^{-1})^{-1}U^2I) & \tilde{\psi}((T^{-1})^{-1}U^2T^{-1})  \end{bmatrix} \\
&=  \begin{bmatrix} \tilde{\psi}(U^2) & \tilde{\psi} (U^2 T^{-1}) \\
						\tilde{\psi} (T U^{2}) &  \tilde{\psi}(-T^{2}U^{-2})  \end{bmatrix} \\
&=  \begin{bmatrix} \beta & 0 \\
				    0 & \frac{\alpha}{\beta}  \end{bmatrix}. \end{align*}
We also note that $\rho(-I) = I$ since $\psi(-I) = 1.$
It is also useful to observe that $$\rho(T^2) = \begin{bmatrix} \alpha & 0 \\ 0 & \alpha \end{bmatrix}.$$
If $\xi_{1} \in \Q$ then $\alpha = e^{2 \pi i \xi_{1}}$ is a root of unity and $\rho(T)$ has finite order. 
If $[\Q(\xi_{2}): \Q] = 2$ then $\xi_{2} \not \in \Q$, $\beta = e^{2 \pi i \xi_{2}}$ is not a root of unity, and $\rho(U^2)$ does not have finite order. Thus the image of $\rho$ is infinite. The fact that $\xi_{1} \in \Q$ and $\xi_{2} \not \in \Q$ implies that $\frac{\alpha}{\beta} = e^{2 \pi i (\xi_{1} -\xi_{2})} \neq \beta.$ Therefore $\rho(U^2)$ is not a scalar matrix. \\

\nin We now prove that $\rho$ is irreducible. Suppose that $\rho$ is reducible. Then there exists a one-dimensional $\Gamma_{0}(2)$-invariant subspace of $\C^{2}$, which we denote by $E$. As $\rho(U^2) = \left[ \begin{array} {cc} \beta & 0 \\
				    0 & \frac{\alpha}{\beta}  \end{array} \right]$, $E$ must equal the $\beta$-eigenspace for
 $\rho(U^2)$, which is spanned by $\left [ \begin{array}{c} 1 \\ 0 \end{array} \right],$ or $E$ must equal the
 $\frac{\alpha}{\beta}$-eigenspace of $\rho(U^2)$, which is spanned by  
$\left [ \begin{array}{c} 0 \\ 1 \end{array} \right].$ However, $\rho(T) = \left[ \begin{array} {cc} 0 & 1 \\
				    \alpha & 0 \end{array} \right]$ does not act on the subspace spanned by  $\left [ \begin{array}{c} 1 \\ 0 \end{array} \right]$ nor does it act on the subspace spanned by $\left [ \begin{array}{c} 0 \\ 1 \end{array} \right].$
Therefore $E$ is not a $\Gamma_{0}(2)$-invariant subspace. This is a contradiction and we conclude that $\rho$ is irreducible.  \\

\nin We recall that the vector-valued function $F_{0} := \eta^{-2k_0} F$ is a meromorphic vector-valued modular form for  $\rho_{0} := \rho \otimes \omega^{-k_0}$ where $\omega$ is the character of $\Gamma$ for the modular form $\eta.$ 
We also recall that $\Gamma_{0}(2)$ has two cusps $\infty$ and $S \cdot \infty = 0$, $\textrm{Stab}_{\Gamma_{0}(2)} \infty  = \langle -I, T \rangle$, $\textrm{Stab}_{\Gamma_{0}(2)} 0 = \langle -I, ST^{2}S^{-1} \rangle$, and that $-U^{-2} = ST^{2}S^{-1}.$ \\

\nin We will prove that $[\Q(r): \Q]  =2$ by analyzing the leading exponent of each of the $q_2 = e^{\pi i \tau}$-expansions of the component functions of $F_{0}|S.$ We begin by noting that $$\left(\rho_{0}(ST^2S^{-1}) F_{0} \right)|_{0} S = \rho_{0}(ST^2S^{-1}) (F_{0}|_{0} S) = (F_{0}|_{0} ST^2S^{-1})|_{0} S = (F_{0}|_{0} S)|_{0}{T^{2}}.$$ We have that $\omega(ST^{2}S^{-1}) = \omega(T^2) = e^{\frac{2 \pi i}{3}}$ since $\omega$ is the character of $\Gamma$ and $\omega(T) =e^{\frac{2 \pi i }{6}}.$ Therefore
\begin{align*}    \rho_{0}(ST^2 S^{-1}) &= \omega(ST^{2}S^{-1})^{-k_0} \rho(ST^2S^{-1}) \\ 
                                                            & = e^{\frac{-2 \pi i k_0}{3}} \rho(-U^{-2}) \\
                        					         & =   e^{\frac{-2 \pi i k_0}{3}} \begin{bmatrix}e^{-2 \pi i \xi_{2}} & 0  \\   0 & e^{2 \pi i (\xi_{2} - \xi_{1})} \end{bmatrix}   \\
									   & =   \begin{bmatrix}e^{2 \pi i (-\xi_{2} - \frac{k_0}{3})} & 0  \\ 0 & e^{2 \pi i (\xi_{2} - \xi_{1} - \frac{k_0}{3})} \end{bmatrix}.  \end{align*}
Thus $$\rho_{0}(ST^{2}S^{-1}) F_{0}|_{0}S =   \begin{bmatrix}e^{2 \pi i (-\xi_{2} - \frac{k_0}{3})} &  0  \\   0 & e^{2 \pi i (\xi_{2} - \xi_{1} - \frac{k_0}{3})} \end{bmatrix}  F_{0}|_{0} S = (F_{0}|_{0}S)|_{0} T^2.$$
Let $P_1$ and $P_2$ denote the functions for which $F_{0}|_{0}S = \left [\begin{matrix} P_1 \\ P_2 \end{matrix} \right].$  Thus $P_{1}(\tau + 2) = e^{2 \pi i (-\xi_{2} - \frac{k_0}{3})}  P_{1}(\tau)$ and 
$P_2(\tau + 2) = e^{2 \pi i (\xi_{2} - \xi_{1} - \frac{k_0}{3})} P_2(\tau).$ We let 
$\widehat{P_{1}}(\tau) := e^{- \pi i (-\xi_{2} - \frac{k_0}{3}) \tau }P_1(\tau)$ and \newline 
\nin $\widehat{P_{2}}(\tau) := e^{- \pi i (\xi_{2} - \xi_{1} - \frac{k_0}{3}) \tau }P_2(\tau).$
Then $\widehat{P_{1}}(\tau + 2)  = \widehat{P_{1}}(\tau)$ and $\widehat{P_{2}}(\tau + 2) = \widehat{P_{2}}(\tau).$ \\

\nin Hence $\widehat{P_1}$ has a Fourier expansion and we write $\widehat{P_{1}}(\tau) = \sum_{n \in \Z} a_n q_2^{n}$
where $q_{2} = e^{\pi i \tau}.$ Thus $P_1(\tau) = q_{2}^{-\xi_{2} - \frac{k_0}{3}} \left(\sum_{n \in \Z} a_n q_2^{n} \right).$ 
The fact that $F_{0}$ is meromorphic at the cusp $0$ implies that $a_n = 0$ if $n < < 0.$ Similarly,  $P_2(\tau) = q_2^{\xi_{2} - \xi_{1} - \frac{k_0}{3}} \left(\sum_{n \in \Z} b_n q_{2}^{n} \right)$ where $b_n = 0$ if $n < < 0.$
Let $r_1$ denote the complex number for which $P_1 = q_2^{r_1}\left(\sum_{n \geq 0} c_n q_2^{n}\right)$ with $c_0 \neq 0$
and let $r_2$ denote the complex number for which $P_2 = q_{2}^{r_{2}}\left(\sum_{n \geq 0} d_n q_{2}^{n}\right)$ with $d_0 \neq 0.$
We have that $r_1 \equiv -\xi_{2} - \frac{k_0}{3}\;  (\textrm{mod } \Z)$ and $r_2 \equiv \xi_2 - \xi_1 - \frac{k_0}{3} \; (\textrm{mod } \Z).$
As $\xi_1 \in \Q$ and $k_0 \in \Z$, we have that $\Q(r_1) = \Q(r_2) = \Q(\xi_2).$ \\

\nin We will now prove that $r_1$ and $r_2$ are the roots of the quadratic polynomial $z^2 + z(-a-\frac{1}{3}) + b + 4c.$ This fact implies that $r$, which is a root of the quadratic polynomial $z^2 + z(-a-\frac{1}{3}) + b + 4c,$ is an algebraic number of degree two and that $\Q(r) = \Q(\xi_2)$ since $\Q(r_1) = \Q(r_2) = \Q(\xi_2).$ We recall that $F_{0}$ satisfies the differential equation 
$$D_0^{2} F_{0} + aG D_{0} F_{0} + (bG^2 + cE_4)F_{0} = 0.$$ We analyze the above differential equation 
at the cusp $S \cdot \infty = 0$ of $\Gamma_{0}(2).$ We have that 
\begin{align*}
0 & =  \left( D_0^{2} F_{0} + aG D_{0} F_{0} + (bG^2 + cE_4)F_{0} \right)|_{4} S \\
& =  D_{0}^{2}(F_{0}|_{0}S) + a G|_{2} S D_{0}(F_{0}|_{0}S) + (bG^{2}|_{4}S + cE_4) F_{0}|_{0}S.
\end{align*}

\nin Let $\mathcal{L}$ denote the differential operator for which $$\mathcal{L} f := D_{0}^{2}(f) + a G|_{2} S D_{0}(f) + (bG^{2}|_{4}S + cE_4)f.$$ Then $0 = \mathcal{L} P_1 = \mathcal{L} P_2.$
We recall that $P_1 = q_{2}^{r_1}\left(\sum_{n \geq 0} c_n q_2^{n} \right)$ where $c_0 \neq 0.$
The fact that $\mathcal{L}P_1 = 0$ implies that the Fourier coefficient of $q_2^{r_1}$ in $\mathcal{L}(q_2^{r_1})$, which we shall compute explicitly,  is zero. \\

\nin The chain rule gives us that $D_0 = \frac{1}{2 \pi i}\frac{d}{d \tau}  = \frac{1}{2} q_{2} \frac{d}{d q_2}.$
Therefore $D_0(q_2^{r_1})= \frac{r_1}{2} q_2^{r_1}.$ We also have: 
\begin{align*} D_2(D_0 (q_2^{r_1})) &= \frac{r_1}{2} D_2(q_2^{r_1}) \\
&= \frac{r_1}{2} \left(D_0(q_2^{r_1}) - \frac{1}{6}E_2 q_2^{r_1} \right) \\
& = \frac{r_1}{2} \left(\frac{r_1}{2}q_2^{r_1} - \frac{1}{6}q_2^{r_1} + O\left(q_2^{r_1 +1}\right)\right) \\
& = \left(\frac{r_1^{2}}{4} - \frac{r_1}{12}\right)q_2^{r_1} + O \left( q_2^{r_1 +1} \right).
\end{align*} 

Thus \begin{align*} \mathcal{L}(q_2^{r_1}) &=  D_2(D_0 (q_2^{r_1})) + aG|_2 S D_0(q_2^{r_1}) + (bG^2|_{4}S + cE_4)q_2^{r_1} \\
& =      \left( \frac{r_1^{2}}{4} - \frac{r_1}{12} \right) q_2^{r_1} + O\left(q_2^{r_1 +1}\right) + a \left(\frac{-1}{2} + O(q_2) \right)\frac{r_1}{2} q_2^{r_1} + \left (\frac{b}{4} + c + O\left(q_2 \right) \right )q_2^{r_1} \\
&=  q_2^{r_1} \left(O\left(q_2\right) + \frac{r_1^{2}}{4} - \frac{r_1}{12}  - \frac{ar_1}{4}   + \frac{b}{4} + c  \right).
\end{align*}

\nin  The Fourier coefficient of $q_2^{r_1}$ in $\mathcal{L}(q_2^{r_1})$ is zero and thus $r_1^2  + (-\frac{1}{3} - a)r_1  + b + 4c = 0.$  Similarly, the Fourier coefficient of $q_2^{r_2}$ in $\mathcal{L}(q_2^{r_2})$ is zero and therefore $r_2$ is also a root of the quadratic polynomial $z^2 + z(-\frac{1}{3}-a) + b + 4c.$ Hence $\Q(r) = \Q(r_1) = \Q(r_2).$ We have previously shown that $\Q(r_1) = \Q(\xi_2)$ and we conclude that $\Q(r) = \Q(\xi_2).$ We explained on page 24 that $\rho(T)$ having finite order implies that $a, b +c \in \Q.$  The fact that $[\Q(r): \Q] =2$ implies that $r \overline{r} = b+ 4c \in \Q.$ Thus $a, b+c , b+4c \in \Q$. Hence $a, b,c \in \Q.$ 
\end{proof}

\nin We now give the proof of \cref{thm: inducedmain}.

\inducedmain*
 
\nin \begin{remark} The group $\Gamma(2)/\langle -I \rangle$ is freely generated by 
$T^2 \langle -I \rangle$ and $U^2 \langle -I \rangle.$ Thus for any $\xi_1, \xi_2 \in \C,$ there exists a character $\psi$ of $\Gamma(2)$ for which $e^{2 \pi i \xi_1} = \psi(T^2), e^{2 \pi i \xi_2}= \psi(U^2),$ and $1 = \psi(-I).$  \end{remark} 

\nin \begin{remark} We have proven in \cref{thm: examples} that  if $\psi$ is a character of $\Gamma(2)$ which satisfies the hypotheses of \cref{thm: inducedmain} then $\rho = \textrm{Ind}_{\Gamma(2)}^{\Gamma_{0}(2)} \psi$ is irreducible, 
$\rho(T)$ has finite order, and $r$ is an algebraic number of degree two. It therefore follows from \cref{thm: algebraicbasis} that if $\rho$ is induced from a character of $\Gamma(2)$ which satisfies the hypotheses of \cref{thm: inducedmain} then for every $k \in \Z$, there is a basis of $M_{k}(\rho')$ consisting of vector-valued modular forms whose component functions have the property that all of their Fourier coefficients are elements of $\Q(r) = \Q(\xi_2).$ \end{remark}

\begin{proof} We have proven in \cref{thm: examples} that if $\psi$ is a character of $\Gamma(2)$ which satisfies the hypotheses of \cref{thm: inducedmain} then $\rho = \textrm{Ind}_{\Gamma(2)}^{\Gamma_{0}(2)} \psi$ satisfies the hypotheses of \cref{thm: oldinducedmain}, $c \in \Q$, and $[\Q(r):\Q] =2$. Moreover, $\Q(r) = \Q(\xi_2) = \Q(\sqrt{M})$. The conclusion of \cref{thm: inducedmain} now follows from \cref{thm: oldinducedmain}.

\end{proof}

\section{Appendix} \label{appendix}
\subsection{The first and second derivatives of the Hauptmodul $\mathfrak{J}$} 
\nin The purpose of this subsection is to prove Propositions \ref{appendix1} and \ref{appendix2}.
 We begin with the proof of Proposition \ref{appendix1}. We will prove a bit more. Namely, we will show that: 
$$\theta(\mathfrak{J}) = (1 - \mathfrak{J})G = \frac{G(E_4 - 4G^2)}{E_4 - G^2}$$
\begin{proof} (Proof of Proposition \ref{appendix1}.) The derivative of a modular function is a meromorphic modular form of weight two. The differential operator $\theta  = \frac{1}{2 \pi i}\frac{d}{d \tau} = q \frac{d}{dq}$ preserves the order of vanishing of a function at $\infty$. Therefore $\theta(\mathfrak{J})$ has a simple pole at $\infty$. Moreover, $\theta(\mathfrak{J})$ has no poles elsewhere since $\mathfrak{J}$ has no poles elsewhere. As $E_4 - G^2$ has a simple zero at $\infty$,  the function $(E_4 - G^2) \theta(\mathfrak{J})$ is holomorphic and thus $(E_4 - G^2) \theta(\mathfrak{J}) \in M_{6}(\Gamma_{0}(2)) = \C E_4 G \bigoplus \C G^{3}.$ We use the $q$-series expansions of $(E_4 - G^2) \theta(\mathfrak{J}), E_4 G,$ and $G^3$ to conclude that $(E_4 - G^2) \theta(\mathfrak{J}) = E_4 G - 4 G^{3}.$ \\

\nin It now suffices to prove that $\frac{E_4 - 4G^2}{E_4 - G^2} = 1 - \mathfrak{J}$. \\ 

\nin  The modular functions $\mathfrak{J}$ and $\frac{E_4 - 4G^2}{E_4 - G^2}$ both have a unique pole, which is a simple pole at $\infty.$ As $\Gamma_{0}(2)$ is a genus zero subgroup, the Riemann-Roch theorem implies that the dimension of the space of meromorphic functions on $\Gamma_{0}(2) \backslash (\mathfrak{H} \bigcup \mathbb{P}^{1}(\Q))$ which have at most a simple pole at $\infty$ and which are holomorphic elsewhere is two. We can take a basis for this space to be the constant function $1$ and $\mathfrak{J}.$ Therefore there $y,z \in \C$ such that  
$\frac{E_4 - 4G^2}{E_4 - G^2} = y + z\mathfrak{J}$. We compare the coefficients of $q^{-1}$ and $q^{0}$ in the $q$-expansions of $\frac{E_4 - 4G^2}{E_4 - G^2}$ and $\mathfrak{J}$ to conclude that $\frac{E_4 - 4G^2}{E_4 - G^2} = 1 - \mathfrak{J}$. 
\end{proof}

\nin \begin{remark} We have shown in the above proof that $\mathfrak{K} = 64 \mathfrak{J} = q^{-1}(1 + O(q)).$ 
\end{remark}

\nin We need the following Proposition before giving the proof of Proposition \ref{appendix2}.
\begin{proposition}  $$\theta(G) = \frac{1}{6}(E_2 G + E_4 - 2G^2)$$
\end{proposition}
\begin{proof} As $G \in M_{2}(\Gamma_{0}(2))$, $D_{2}(G) = \theta(G) -\frac{1}{6}E_2 G \in M_{4}(\Gamma_{0}(2)) = \C E_4 \bigoplus \C G^2$. We use the $q$-series expansion of $D_2(G), E_4,$ and $G^2$ to conclude that 
$D_2(G) = \theta(G) - \frac{1}{6}E_2 G = -\frac{1}{6} E_4 - \frac{1}{3}G^2.$

\end{proof} 

\nin We will also need to use the fact that $\frac{E_4(\tau)}{G^2(\tau)} = \frac{\mathfrak{J}(\tau) + 3}{\mathfrak{J}(\tau)}$ in our proof of Proposition \ref{appendix2}. 
This fact can be derived from elementary algebra since $\mathfrak{J} := \frac{3G^2}{E_4 - G^2}.$  
We now give the proof of Proposition \ref{appendix2}. Namely, we show that:

$$\theta^{2}(\mathfrak{J}) =G^2(1 - \mathfrak{J})(\frac{-7\mathfrak{J} + 3}{6\mathfrak{J}})  + \frac{1}{6}E_2 \theta(\mathfrak{J}).$$

\begin{proof}  (Proof of Proposition \ref{appendix2}) We will use the formula  $\theta(G) = \frac{1}{6}(E_2 G + E_4 - 2G^2)$ to compute $\theta^{2}(\mathfrak{J}).$
We have that 
\begin{align*}  \theta^2(\mathfrak{J}) &= \theta(G(1 - \mathfrak{J}))  = -\theta(\mathfrak{J})G + (1 - \mathfrak{J})\theta(G) \\
&=  - G^2(1 - \mathfrak{J})     +  \frac{1}{6}(E_2 G + E_4 - 2G^2)(1 - \mathfrak{J}) \\
&= G^2(1 - \mathfrak{J})(-\frac{4}{3} + \frac{E_4}{6G^{2}})  + \frac{1}{6}E_2 G(1 - \mathfrak{J}) \\
& = G^2(1 - \mathfrak{J})(-\frac{4}{3} + \frac{\mathfrak{J} +3}{6\mathfrak{J}})  + \frac{1}{6}E_2 \theta(\mathfrak{J}) \\
& =   G^2(1 - \mathfrak{J})(\frac{-7\mathfrak{J} + 3}{6\mathfrak{J}})  + \frac{1}{6}E_2 \theta(\mathfrak{J}).  \end{align*}
\end{proof}

\subsection{An integrality result} 
\begin{lemma} \label{integral} Let $\mathfrak{K} :=  64 \mathfrak{J} = \frac{192G^2}{E_4 - G^2}.$  Then $\mathfrak{K} \in \frac{1}{q} \Z[[q]]^{\times}.$ 
\end{lemma}

\begin{proof}
We have that  \begin{align*}  G = -E_2(\tau) + 2E_2(2\tau) & = -(1 - 24\sum_{n=1}^{\infty} \sigma(n)q^{n}) + 2(1- 24 \sum_{n=1}^{\infty} \sigma(n)q^{2n}) \\
& =    1 + 24 \sum_{n=0}^{\infty} \sigma(2n+1)q^{2n+1} +  \sum_{n=1}^{\infty} (24 \sigma(2n) - 48 \sigma(n))q^{2n}. \end{align*} Let $P =  \sum_{n=0}^{\infty} \sigma(2n+1)q^{2n+1} + \sum_{n=1}^{\infty} (\sigma(2n) - 2 \sigma(n))q^{2n}.$ Then $G = 1 + 24P$ and $$G^2 = 1 + 48P + 24^2 P^2 \equiv 1 + 48P \; (\textrm{mod } 192).$$
Therefore \begin{align*} G^2 - E_4 & \equiv 48P - 240 \sum_{n=1}^{\infty} \sigma_{3}(n)q^{n} \\ 
& \equiv 48\left(P - \sum_{n=1}^{\infty} \sigma_{3}(n)q^{n}\right) \; (\textrm{mod } 192). \end{align*}
Thus to prove that  $G^2- E_4 \in 192 \Z[[q]]$, it suffices to show that $$\sum_{n=0}^{\infty} \sigma(2n+1)q^{2n+1} + \sum_{n=1}^{\infty} (\sigma(2n) - 2 \sigma(n))q^{2n} = P \equiv  \sum_{n=1}^{\infty} \sigma_{3}(n)q^{n} \; (\textrm{mod } 4).$$ Equivalently, we must show that for all $n \geq 0$, $\sigma(2n +1) \equiv \sigma_3(2n+1) \; (\textrm{mod } 4)$ and that 
for all $n \geq 1$, $\sigma(2n) - 2\sigma(n) \equiv \sigma_{3}(2n) \; (\textrm{mod } 4).$  We observe that since any divisor of an odd integer is odd, $\sigma(2n +1) = \sum_{d \mid 2n +1} d \equiv \sum_{d \mid 2n +1} d^{3} = \sigma_{3}(2n +1)\; (\textrm{mod } 4).$ To prove that  $\sigma(2n) - 2\sigma(n) \equiv \sigma_{3}(2n) \; (\textrm{mod } 4)$, we first write $n = 2^{e} n'$ where $n'$ is odd and $e \geq 0$. We have that 
\begin{align*}\sigma(2n) - 2 \sigma(n) & = \sigma(2^{e+1} n') - 2 \sigma(2^{e} n') \\
& = (\sigma(2^{e+1}) - 2 \sigma(2^{e})) \sigma(n') \\
& = (2^{e+2} - 1 - 2(2^{e+1} - 1)) \sigma(n') \\
& = \sigma(n').
\end{align*} 
Thus $\sigma_{3}(2n) = \sigma_{3}(2^{e+1} n') = \sigma_{3}(2^{e+1}) \sigma_{3}(n') \equiv \sigma_{3}(n') \;(\textrm{mod } 4).$ Because $n'$ is odd, $\sigma_{3}(n') \equiv \sigma(n') \; (\textrm{mod } 4).$ 
 Thus $\sigma_{3}(2n) \equiv \sigma(n') \; (\textrm{mod } 4).$ 
We have thus proven that for all $n \geq 1$, $\sigma_{3}(2n) \equiv \sigma(n') \equiv \sigma(2n) - 2 \sigma(n) \; (\textrm{mod } 4)$. Hence $G^2 - E_4 \in 192 \Z[[q]].$  Moreover, $G^2 - E_4 = O(q)$ and thus $\frac{G^2 - E_4}{192q} \in \Z[[q]].$
We also have that $\frac{G^2 - E_4}{192q}  =  1 + O(q).$ We recall that the elements in $\Z[[q]]^{\times}$ are exactly those elements of $\Z[[q]]$ whose constant term is equal to $1$ or $-1.$
Therefore $\frac{G^2 - E_4}{192q} \in \Z[[q]]^{\times}$ and hence  $\frac{192 q} {G^2 - E_4} \in \Z[[q]]^{\times}.$ 
Finally, $G^2 = 1 + O(q) \in \Z[[q]]^{\times}$ and so $\frac{192G^2q}{E_4 - G^2} = (64q) \mathfrak{J} \in \Z[[q]]^{\times}.$ Thus $\mathfrak{J} \in \frac{1}{64q} \Z[[q]]^{\times}$ and 
$64 \mathfrak{J} = \mathfrak{K} \in \frac{1}{q} \Z[[q]]^{\times}.$
\end{proof}

\subsection{The $q_2$-expansion of $G|_{2}S$.} In this subsection, we show that the value of the constant term of the $q_2$-expansion of $G|_{2}S$ is equal to $\frac{-1}{2}.$ This result is used in the proof of \cref{thm: examples}.  Let $\theta = 1 + 2 \sum_{n=1}^{\infty} q^{n^{2}}$ and let $\mathcal{E} = \frac{\eta(4 \tau)^{8}}{\eta(2 \tau)^{4}}.$
Then $\theta^{4} = 1 + \sum_{n=1}^{\infty} r_{4}(n) q^{n}$ where $r_{4}(n)$ denotes the number of ways 
$n$ can be written as a sum of four squares. 
One may show that $G = \theta^{4} + 16 \mathcal{E}$ by using the theory of modular forms on $\Gamma_{0}(4)$ or by using the results that if $n$ is odd then $r_{4}(n) =8 \sigma_{1}(n)$, and if $n$ is even then $r_{4}(n) = 24 \sigma_{1}(n_0),$ where $n_0$ denotes the odd part of $n.$ \\

\nin We now use the fact that $\eta(\frac{-1}{\tau}) =  \sqrt{\frac{\tau}{i}} \eta(\tau)$ to calculate the constant term of the $q_{2} = e^{\pi i \tau}$-expansion of $\mathcal{E}|_{2}S.$ We compute that
\begin{align*}
\mathcal{E}|_2 S  & = \tau^{-2} \eta \left(4 \left(\frac{-1}{\tau}\right)\right)^{8}\eta \left(2 \left( \frac{-1}{\tau} \right) \right)^{-4}\\
& =\tau^{-2} \eta \left(\left(\frac{-1}{\frac{\tau}{4}}\right)\right)^{8} \eta \left(  \left(\frac{-1}{\frac{\tau}{2}} \right) \right)^{-4} \\
& = \tau^{-2} \left(\sqrt{\frac{\frac{\tau}{4}}{i}} \right)^{8} \eta^{8} \left(\frac{\tau}{4} \right)   \left(\sqrt{\frac{\frac{\tau}{2}}{i}} \right)^{-4} \eta^{-4} \left(\frac{\tau}{2} \right) \\
& = \frac{-1}{64}  \eta^{8} \left(\frac{\tau}{4} \right)  \eta^{-4} \left(\frac{\tau}{2} \right) \\
& = \frac{-1}{64} + \cdots.
\end{align*}
We will use the transformation law $\theta(\frac{-1}{\tau}) = \sqrt{\frac{\tau}{2i}} \theta(\frac{\tau}{4})$ to compute 
the $q_2$-expansion of $\theta^{4}|_{2} S.$
We have that 
\begin{align*} \theta^{4}|_{2} S & = \tau^{-2} \theta^{4} \left (\frac{-1}{\tau} \right)\\
& =  \tau^{-2} \left( \sqrt{\frac{\tau}{2i}} \right)^{4} \theta^{4} \left (\frac{\tau}{4} \right) \\
&  = \frac{-1}{4} \theta^{4} \left (\frac{\tau} {4} \right) \\
& = -\frac{1}{4} + \cdots.       \end{align*}
Therefore \begin{align*} G|_{2}S& = (\theta^{4} + 16 \mathcal{E})|_{2} S\\
&  = \theta^{4}|_{2} S + 16 \mathcal{E}|_{2} S \\
&=\left( -\frac{1}{4} + \cdots \right)+ 16 \left(\frac{-1}{64} + \cdots \right)\\
& = -\frac{1}{2} + \cdots. \end{align*}

\providecommand{\bysame}{\leavevmode\hbox to3em{\hrulefill}\thinspace}
\providecommand{\MR}{\relax\ifhmode\unskip\space\fi MR }
\providecommand{\MRhref}[2]{%
  \href{http://www.ams.org/mathscinet-getitem?mr=#1}{#2}
}
\providecommand{\href}[2]{#2}

\end{document}